\documentclass[11pt,leqno]{amsart}

\usepackage{latexsym}
\usepackage{fullpage}
\usepackage{amssymb}
\usepackage{amscd}
\usepackage{float}
\usepackage{graphicx}
\usepackage{amsfonts}
\usepackage{pb-diagram}
 \usepackage{amsmath,amscd}

\newtheorem{thm}{Theorem}
\newtheorem{cor}{Corollary}
\newtheorem{lemma}{Lemma}
\newtheorem{prop}{Proposition}
\newtheorem{defn}{Definition}
\newtheorem{remark}{Remark}

\newtheorem*{ex}{Generic Example}
\newtheorem{nt}{Notation}

\begin{document}

\title[The Kauffman bracket skein module of the complement of $(2, 2p+1)$-torus knots via braids]
  {The Kauffman bracket skein module of the complement of $(2, 2p+1)$-torus knots via braids}

\author{Ioannis Diamantis}
\address{ International College Beijing,
China Agricultural University,
No.17 Qinghua East Road, Haidian District,
Beijing, {100083}, P. R. China.}
\email{ioannis.diamantis@hotmail.com}



\setcounter{section}{-1}

\date{}

\begin{abstract}
In this paper we compute the Kauffman bracket skein module of the complement of $(2, 2p+1)$-torus knots, $KBSM(T_{(2, 2p+1)}^c)$, via braids. We start by considering geometric mixed braids in $S^3$, the closure of which are mixed links in $S^3$ that represent links in the complement of $(2, 2p+1)$-torus knots, $T_{(2, 2p+1)}^c$. Using the technique of parting and combing geometric mixed braids, we obtain algebraic mixed braids, that is, mixed braids that belong to the mixed braid group $B_{2, n}$ and that are followed by their ``coset'' part, that represents $T_{(2, 2p+1)}^c$. In that way we show that links in $T_{(2, 2p+1)}^c$ may be pushed to the genus 2 handlebody, $H_2$, and we establish a relation between $KBSM(T_{(2, 2p+1)}^c)$ and $KBSM(H_2)$. In particular, we show that in order to compute $KBSM(T_{(2, 2p+1)}^c)$ it suffices to consider a basis of $KBSM(H_2)$ and study the effect of combing on elements in this basis. We consider the standard basis of $KBSM(H_2)$ and we show how to treat its elements in $KBSM(T_{(2, 2p+1)}^c)$, passing through many different spanning sets for $KBSM(T_{(2, 2p+1)}^c)$. These spanning sets form the intermediate steps in order to reach at the set $\mathcal{B}_{T_{(2, 2p+1)}^c}$, which, using an ordering relation and the notion of total winding, we prove that it forms a basis for $KBSM(T_{(2, 2p+1)}^c)$. Note that elements in $\mathcal{B}_{T_{(2, 2p+1)}^c}$ have no crossings on the level of braids, and in that sense, $\mathcal{B}_{T_{(2, 2p+1)}^c}$ forms a more natural basis of $KBSM(T_{(2, 2p+1)}^c)$ in our setting. We finally consider c.c.o. 3-manifolds $M$ obtained from $S^3$ by surgery along the trefoil knot and we discuss steps needed in order to compute the Kauffman bracket skein module of $M$. We first demonstrate the process described before for computing the Kauffman bracket skein module of the complement of the trefoil, $KBSM(Tr^c)$, and we study the effect of braid band moves on elements in the basis of $KBSM(Tr^c)$. These moves reflect isotopy in $M$ and are similar to the second Kirby moves.
\smallbreak
The ``braid'' method that we propose for computing Kauffman bracket skein modules seem promising in computing KBSM of arbitrary c.c.o. 3-manifolds $M$. The only difficulty lies in finding the sufficient relations that reduce elements in the basis of our underlying genus g-handlebody, $H_g$. These relations come from combing in the case of knot complements and from combing and braid band moves for 3-manifolds obtained by surgery along a knot in $S^3$. Our aim is to set the necessary background of this ``braid'' approach in order to compute Kauffman bracket skein modules of arbitrary 3-manifolds.
\bigbreak
\noindent 2020 {\it Mathematics Subject Classiffication.} 57K10, 57K12, 57K14, 57K35, 57K45, 57K99, 20F36, 20F38, 20C08.
\bigbreak
\end{abstract}

\maketitle

\begin{center}
\begin{minipage}{10cm}
{\scriptsize\tableofcontents}
\end{minipage}
\end{center}

\section{Introduction and overview}\label{intro}

Skein modules were independently introduced by Przytycki \cite{P} and Turaev \cite{Tu} as generalizations of knot polynomials in $S^3$ to knot polynomials in arbitrary 3-manifolds. The essence is that skein modules are quotients of free modules over ambient isotopy classes of links in 3-manifolds by properly chosen local (skein) relations.

\begin{defn}\rm
Let $M$ be an oriented $3$-manifold and $\mathcal{L}_{{\rm fr}}$ be the set of isotopy classes of unoriented framed links in $M$. Let $R=\mathbb{Z}[A^{\pm1}]$ be the Laurent polynomials in $A$ and let $R\mathcal{L}_{{\rm fr}}$ be the free $R$-module generated by $\mathcal{L}_{{\rm fr}}$. Let $\mathcal{S}$ be the ideal generated by the skein expressions $L-AL_{\infty}-A^{-1}L_{0}$ and $L \bigsqcup {\rm O} - (-A^2-A^{-2})L$, where $L_{\infty}$ and $L_{0}$ are represented schematically by the illustrations in Figure~\ref{skein}. Note that blackboard framing is assumed and that $L \bigsqcup {\rm O}$ stands for the union of a link $L$ and the trivial framed knot in a ball disjoint from $L$. 

\begin{figure}[!ht]
\begin{center}
\includegraphics[width=1.9in]{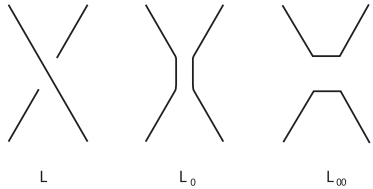}
\end{center}
\caption{The links $L_{\infty}$ and $L_{0}$ locally.}
\label{skein}
\end{figure}

\noindent Then the {\it Kauffman bracket skein module} of $M$, KBSM$(M)$, is defined to be:

\begin{equation*}
{\rm KBSM} \left(M\right)={\raise0.7ex\hbox{$
R\mathcal{L}_{{\rm fr}} $}\!\mathord{\left/ {\vphantom {R\mathcal{L_{{\rm fr}}} {\mathcal{S} }}} \right. \kern-\nulldelimiterspace}\!\lower0.7ex\hbox{$ S  $}}.
\end{equation*}
\end{defn}

In \cite{P} the Kauffman bracket skein module of the handlebody of genus 2, $H_2$, is computed using diagrammatic methods by means of the following theorem:

\begin{thm}[\cite{P}]\label{tprz}
The Kauffman bracket skein module of $H_2$, KBSM($H_2$), is freely generated by an infinite set of generators $\left\{x^i\, y^j\, z^k,\ (i, j, k)\in \mathbb{N}\times \mathbb{N}\times \mathbb{N}\right\}$, where $x^i\, y^j\, z^k$ is shown in Figure~\ref{przt}.
\end{thm}

\noindent Note that in \cite{P}, $H_2$ is represented by a twice punctured plane and diagrams in $H_2$ determine a framing of links in $H_2$. The two dots in Figure~\ref{przt} represent the punctures.

\begin{figure}[!ht]
\begin{center}
\includegraphics[width=2.6in]{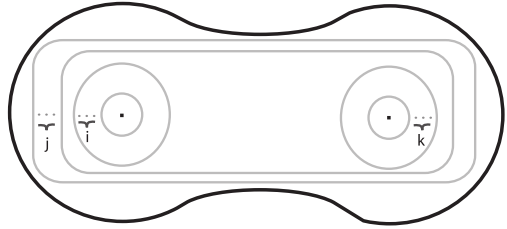}
\end{center}
\caption{The basis of KBSM($H_2$), $B_{H_2}$.}
\label{przt}
\end{figure}

In \cite{B} the author computed the Kauffman bracket skein module of the complement of $(2, 2p+1)$-torus knots diagrammatically. More precisely, the author considered the complement of the $(2, 2p+1)$-torus knots obtained by attaching a 2-handle on $H_2$, and using the technique developed in \cite{HP}, the author studied the effect of this 2-handle to the $KBSM(H_2)$. In particular he identified the relations added to $KBSM(H_2)$ by attaching the 2-handle and he showed that the Kauffman bracket skein module of the complement $(2, 2p+1)$-torus knots is generated by the following set:

\begin{equation}
B_{T_{(2, 2p+1)}^c}\ =\ \left\{(a\, , b),\, |\, b\leq p \right\}\, \cup\, \{U\},
\end{equation}

\noindent where $U$ denotes the unknot and $(a, b)$ denotes the link in $T_{(2, 2p+1)}^c$ consisting of $a$ copies of the curve $x$ and at most one copy of the curve $y$. Elements of this basis are illustrated in Figure~\ref{bul} for the case of the complement of the $(2, 3)$-torus knot, i.e. the complement of the trefoil knot, $Tr^c$.

\begin{figure}[!ht]
\begin{center}
\includegraphics[width=1.4in]{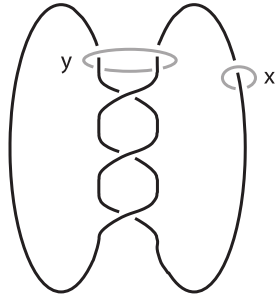}
\end{center}
\caption{Elements in the basis of KBSM($Tr^c$), $B_{Tr^c}$.}
\label{bul}
\end{figure}

\noindent The most challenging and technical part in this method is the identification of the sufficient relations to eliminate elements in $B_{H_2}$ and obtain the basis $B_{T_{(2, 2p+1)}^c}$.

\bigbreak

In this paper we propose an alternative method for the computation of the Kauffman bracket skein module of the complement of $(2, 2p+1)$-torus knots (and arbitrary $3$-manifolds in general) via braids. More precisely, we work on open braided form of links in $T_{(2, 2p+1)}^c$ and we present an alternative basis of $KBSM(T_{(2, 2p+1)}^c)$, which follows from $KBSM(H_2)$ more ``naturally'' in the braid setting. We start by presenting the complement of $(2, 2p+1)$-torus knots by closed braids in $S^3$ and links in $T_{(2, 2p+1)}^c$ by mixed links in $S^3$. Using the technique of parting and combing, we first isotope the links in $T_{(2, 2p+1)}^c$ to links in $H_2$ and we then consider the images of these links in the $KBSM(H_2)$. We relate the $KBSM(T_{(2, 2p+1)}^c)$ to the $KBSM(H_2)$ and in particular and we show that in order to compute $KBSM(T_{(2, 2p+1)}^c)$, it suffices to consider the effect of combing (as well as the Kauffman bracket skein relations) to elements in the basis $B_{H_2}$. We pass through some spanning sets of $KBSM(T_{(2, 2p+1)}^c)$ until we reach at the basis $\mathcal{B}_{T_{(2, 2p+1)}^c}$ of $KBSM(T_{(2, 2p+1)}^c)$.
\smallbreak
It is worth mentioning that the braid approach for computing Kauffman bracket skein modules of 3-manifolds has been successfully applied for the case of the Solid Torus in \cite{D} and for the case of the handlebody of genus 2 in \cite{D1}. Moreover, the braid approach has been applied for the computation of HOMFLYPT skein modules of 3-manifolds (see \cite{DL2, DL3, DL4, DLP, D5}). Finally, the same approach is to be used for the computation of skein modules of other knotted objects, such as tied-links (\cite{D3, D3}), pseudo-links and tied pseudo-links (\cite{D2, D6}). 

\bigbreak

The paper is organized as follows: In \S~\ref{topo} we set up the appropriate topological setting by representing knots and links in the complement of a $(2, 2p+1)$-torus knot by mixed links in $S^3$ and we recall isotopy for mixed links in $S^3$. We then translate isotopy of mixed links to braid equivalence by introducing first the notion of geometric mixed braids. Applying the techniques of parting and combing on geometric mixed braids, we obtain algebraic mixed braids, that is, elements of the mixed braid group $B_{2, n}$. In \S~\ref{h2totc} we relate the Kauffman bracket skein module of the genus 2 handlebody $H_2$ to the Kauffman bracket skein module of the complement of $(2, 2p+1)$-torus knots. More precisely, in \S~\ref{h2kt} we recall the knot theory of $H_2$ and in \S~\ref{kbsmh2} we present some basic sets of $KBSM(H_2)$ via braids. In \S~\ref{sstc} we start from a well known basis of $KBSM(H_2)$ and we study the effect of combing to elements in this basis. We pass through many different spanning sets of $KBSM(T_{(2, 2p+1)}^c)$ until we reach at the spanning set $\mathcal{B}_{T_{(2, 2p+1)}^c}$ in \S~\ref{basTrc}. Finally, in \S~\ref{linind} we prove that elements in $\mathcal{B}_{T_{(2, 2p+1)}^c}$ are linear independent with the use of the notion of total winding, that we also introduce in theis subsection. Moreover, in \S~\ref{trefc} we demonstrate the braid technique for computing Kauffman bracket skein modules of complements of torus knots for the case of the complement of the trefoil knot. Then, in \S~\ref{trefcs} we describe the main ideas of the braid approach toward the computation of the Kauffman bracket skein module of c.c.o. 3-manifolds $M$, obtained  from $S^3$ by integral surgery along the trefoil knot. We start by relating $KBSM(M)$ to $KBSM(Tr^c)$ and we then study the effect of (braid) band moves and combing on elements in $KBSM(Tr^c)$. The (braid) band moves describe isotopy in $M$, but not in $Tr^c$. Finally, we present the idea for showing that $KBSM(M)$ is finitely generated via the braid technique.

\section{Topological Set Up}\label{topo}

In this section we recall the topological setting from \cite{LR1, LR2}, and \cite{DL1}. Throughout the paper, we will denote the complement of $(2, 2p+1)$-torus knots by $T_{(2, 2p+1)}^c\, :=\, S^3\backslash T_{(2, 2p+1)}$, where $T_{(2, 2p+1)}$ denotes a $(2, 2p+1)$-torus knot. Finally, we will be using the case of the complement of the trefoil knot for some illustrations and we will denote that by $Tr^c$.

\subsection{Mixed Links and Isotopy}

$T_{(2, 2p+1)}^c$ is homeomorphic to $S^3\backslash \widehat{B}$, where $\widehat{B}$ is the closure of the braid $B$ which is isotopic to $T_{(2, 2p+1)}$. Consider now a link $L$ in $S^3\backslash \widehat{B}$ and fix $\hat{B}$ pointwise. As explained in \cite{LR1, LR2, DL1}, the link $L$ can be represented by a mixed link $\widehat{B}\cup L$ in $S^3$, that is, a link in $S^3$ consisting of two parts: the {\it fixed part} $\widehat{B}$ that represents $T_{(2, 2p+1)}^c$ and the moving part that represents the link $L$ in $T_{(2, 2p+1)}^c$. A \textit{mixed link diagram }is a diagram $\widehat{B}\cup \widetilde{L}$ of $\widehat{B}\cup L$ on the plane of $\widehat{B}$, where this plane is equipped with the top-to-bottom direction of the braid $B$. See Figure~\ref{mixl} for an example of a mixed link that represents a link in $Tr^c$.

\begin{figure}[!ht]
\begin{center}
\includegraphics[width=1.5in]{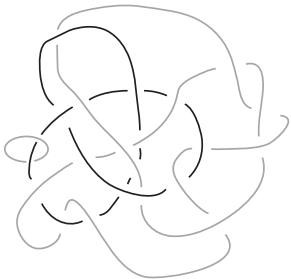}
\end{center}
\caption{A mixed link in $S^3$ representing a link in $Tr^c$.}
\label{mixl}
\end{figure}

Consider now two isotopic links $L_1, L_2$ in $T_{(2, 2p+1)}^c$. Corresponding mixed links if theirs, $\widehat{B}\cup L_1$ and $\widehat{B}\cup L_2$, will differ by a finite sequence of the classical Reidemeister moves for the moving part of the mixed links, and the mixed Reidemeister moves illustrated in Figure~\ref{mr}, that involve the fixed and the moving part of the mixed links.

\begin{figure}[!ht]
\begin{center}
\includegraphics[width=2.8in]{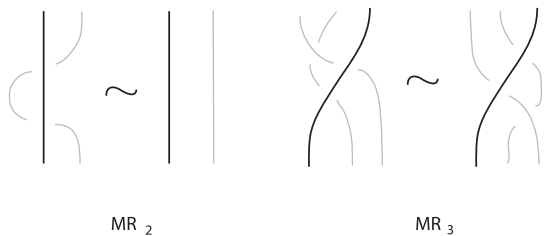}
\end{center}
\caption{Mixed Redimeister moves.}
\label{mr}
\end{figure}

\subsection{Mixed Braids}\label{mixedbr}

In order to translate isotopy of links in $T_{(2, 2p+1)}^c$ into braid equivalence, we need to introduce first the notion of geometric
mixed braids. A \textit{geometric mixed braid} is an element of the group $B_{2+n}$, where 2 strands form the fixed part of the mixed link representing the complement of the $(2, 2p+1)$-torus knot, and $n$ strands form the {\it moving subbraid\/} $\beta$ representing the link $L$ in $T_{(2, 2p+1)}^c$. For an illustration see Figure~\ref{mb} for the case of $Tr^c$. By the Alexander theorem for knots and links in $T_{(2, 2p+1)}^c$ (cf. Thm.~5.4 \cite{LR1}), a mixed link diagram $\widehat{B}\cup \widetilde{L}$ of $\widehat{B}\cup L$ may be turned into a geometric mixed braid $B\cup \beta$ with isotopic closure.

\begin{figure}[!ht]
\begin{center}
\includegraphics[width=0.7in]{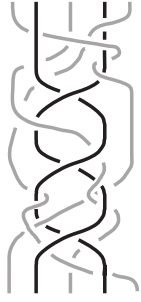}
\end{center}
\caption{A geometric mixed braid in $S^3$.}
\label{mb}
\end{figure}

We now {\it part} the geometric mixed braids in order to separate its strands into two sets: the strands of the fixed subbraid $B$ representing the complement of the $(2, 2p+1)$-torus knot, and the moving strands of the braid representing a link in $T_{(2, 2p+1)}^c$. More precisely:

\begin{defn}\rm
{\it Parting} a geometric mixed braid $B \bigcup\beta$ on $2+n$ strands means to separate its endpoints into two different sets, the first $2$ belonging to the subbraid $B$ and the last $n$ to $\beta$, and so that the resulting braids have isotopic closures. This can be realized by pulling each pair of corresponding moving strands to the right and {\it over\/} or {\it under\/} each  strand of $B$ that lies on their right. We start from the rightmost pair respecting the position of the endpoints. The result of parting is a {\it parted mixed braid}. If the strands are pulled always over the strands of $B$, then this parting is called {\it standard parting}. See the middle illustration of Figure~\ref{partandcomb} for the standard parting of an abstract mixed braid.
\end{defn}

For more details on the technique of parting the reader is referred to \cite{LR2, DL1}. 

\bigbreak

We {\it comb} the parted mixed braid in order to separate the braiding of the fixed subbraid $B$ from the braiding of the moving strands (see Figure~\ref{lo} and the right most illustration of Figure~\ref{partandcomb} for an abstract illustration). After combing a parted mixed braid we obtain an \textit{algebraic mixed braid}. 

\begin{defn}\rm
An \textit{algebraic mixed braid} is a mixed braid on $2+n$ strands such that the first $2$ strands are fixed and form the identity braid on $2$ strands $I_2$ and the next $n$ strands are moving strands and represent a link in $T_{(2, 2p+1)}^c$. The set of all algebraic mixed braids on $2+n$ strands forms a subgroup of $B_{2+n}$, denoted $B_{2,n}$, called the {\it mixed braid group}. 
\end{defn}

The mixed braid group $B_{2,n}$ has been introduced and studied in \cite{La1} (see also \cite{OL}) and it is shown that it has the following presentations:

\begin{center} \label{B}
$
B_{2,n} = \left< \begin{array}{ll}  \begin{array}{l}
t, \tau,  \\
\sigma_1, \ldots ,\sigma_{n-1}  \\
\end{array} &
\left|
\begin{array}{l} \sigma_k \sigma_j=\sigma_j \sigma_k, \ \ |k-j|>1   \\
\sigma_k \sigma_{k+1} \sigma_k = \sigma_{k+1} \sigma_k \sigma_{k+1}, \ \  1 \leq k \leq n-1  \\
t \sigma_k = \sigma_k t, \ \ k \geq 2   \\
\tau \sigma_k = \sigma_k \tau, \ \ k \geq 2    \\
 \tau \sigma_1 \tau \sigma_1 = \sigma_1 \tau \sigma_1 \tau \\
 t \sigma_1 t \sigma_1 = \sigma_1 t \sigma_1 t \\
 \tau (\sigma_1 t {\sigma^{-1}_1}) =  (\sigma_1 t {\sigma^{-1}_1})  \tau
\end{array} \right.  \end{array} \right>,
$
\end{center}

or

\begin{center} \label{B2}
$
B_{2,n} = \left< \begin{array}{ll}  \begin{array}{l}
t, T,  \\
\sigma_1, \ldots ,\sigma_{n-1}  \\
\end{array} &
\left|
\begin{array}{l} \sigma_k \sigma_j=\sigma_j \sigma_k, \ \ |k-j|>1   \\
\sigma_k \sigma_{k+1} \sigma_k = \sigma_{k+1} \sigma_k \sigma_{k+1}, \ \  1 \leq k \leq n-1  \\
t \sigma_k = \sigma_k t, \ \ k \geq 2    \\
T \sigma_k = \sigma_k T, \ \ k \geq 2    \\
 t \sigma_1 T \sigma_1 = \sigma_1 T \sigma_1 t  \\
\end{array} \right.  \end{array} \right>,
$
\end{center}

\noindent where the {\it loop generators} $t, \tau, T$ are as illustrated in Figure~\ref{Loops}.

\begin{figure}[!ht]
\begin{center}
\includegraphics[width=5.2in]{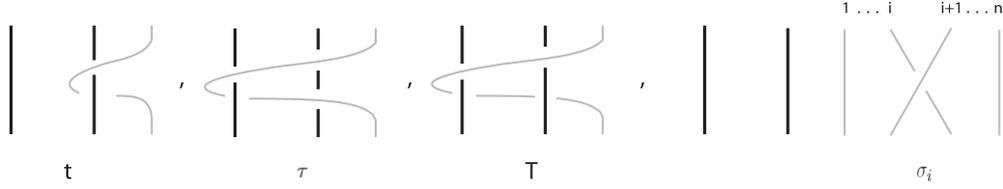}
\end{center}
\caption{ The loop generators $t, \tau$ and $T$ and the braiding generators $\sigma_i$ of $B_{2,n}$. }
\label{Loops}
\end{figure}

Let now $\Sigma_1$ denote the crossing between the $1^{st}$ and the $2^{nd}$ strand of the fixed subbraid. Then, for all $j=1,\ldots,n-1$ we have: $\Sigma_1 \sigma_j = \sigma_j \Sigma_1$. Thus, the only generating elements of the moving part that are affected by combing are the loops $t, \tau$ and $T$. This is illustrated in Figure~\ref{lo}. 

\begin{figure}[!ht]
\begin{center}
\includegraphics[width=5.3in]{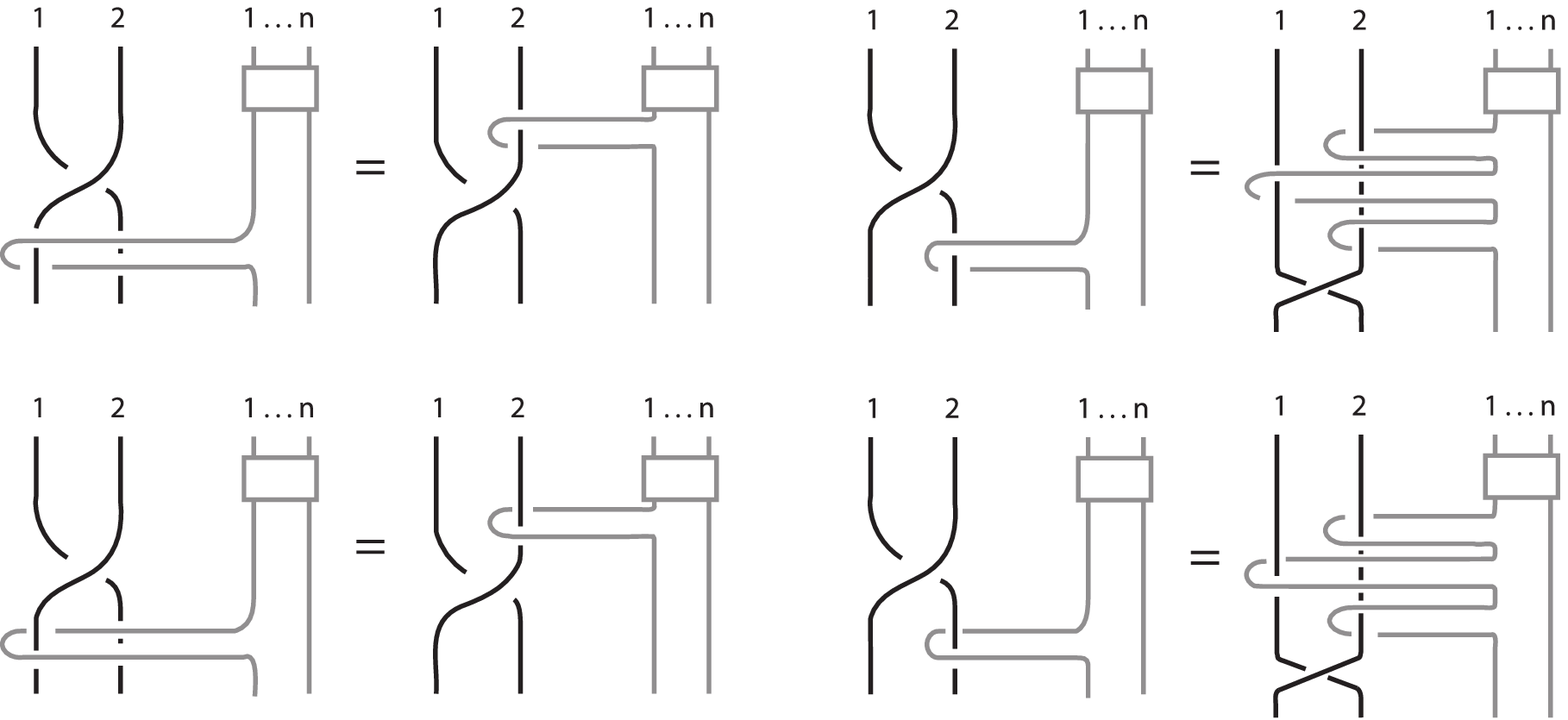}
\end{center}
\caption{Combing.}
\label{lo}
\end{figure}

The effect of combing a parted mixed braid is to separate it into two distinct parts:
the {\it algebraic part} at the top, which has all fixed strands forming the identity braid, so it is an element of the mixed braid group $B_{2,n}$, and which contains all the knotting and linking information of the link $L$ in $T_{(2, 2p+1)}^c$; and
the {\it coset part} at the bottom, which contains only the fixed subbraid $B$ and an identity braid for the moving part (see rightmost illustration in Figure~\ref{partandcomb}).

\begin{figure}[!ht]
\begin{center}
\includegraphics[width=5.3in]{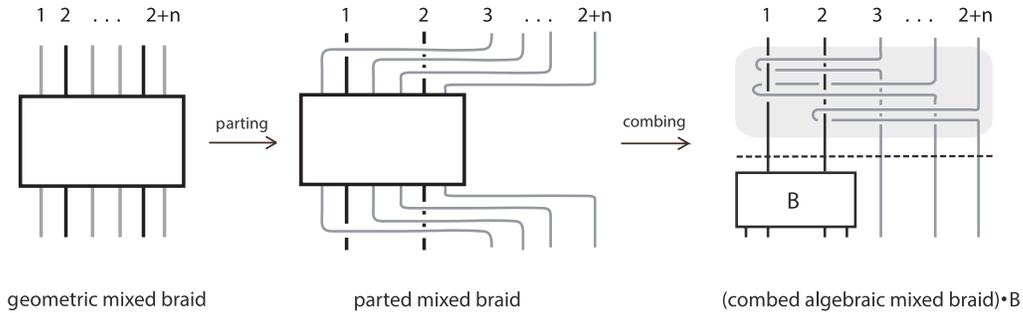}
\end{center}
\caption{Parting and combing a geometric mixed braid.}
\label{partandcomb}
\end{figure}

\begin{remark}\rm
As noted in \cite{La1, DL1}, if we denote the set of parted mixed braids on $n$ moving strands with fixed subbraid $B$ by $C_{2,n}$, then concatenating two elements of $C_{2,n}$ is not a closed operation since it alters the description of the manifold. However, as a result of combing, the set $C_{2,n}$ is a coset of $B_{2,n}$ in $B_{2+n}$ characterized by the fixed subbraid $B$.
\end{remark}

The group $B_{m,n}$ embeds naturally into the group $B_{m,n+1}$. We shall denote

$$B_{2,\infty}:=\bigcup_{n=1}^{\infty }B_{2,n} \quad {\rm and \ similarly} \quad C_{2,\infty}=\bigcup_{n=1}^{\infty }C_{2,n}.$$

\begin{thm}\label{algcco} \
Two  oriented links in  $T_{(2, 2p+1)}^c$ are isotopic if and only if any two corresponding algebraic mixed braid representatives in  $B_{2,\infty}$  differ by a finite sequence of the following moves:

\[
\begin{array}{llll}
1. & Stabilization\ moves: & {\alpha} {\beta}\, \sim\, \alpha\, {\sigma^{\pm 1}_n}\, {\beta}, & {\rm for}\ \alpha, \beta  \in B_{2,n},\\
&&&\\
2. & Conjugation: & \alpha\, \sim\, {\sigma^{\pm 1}_j} \alpha {\sigma^{\mp 1}_j}, & {\rm for}\ \alpha, \sigma_j  \in B_{2,n},\\
&&&\\
3. & Combed\ loop\ conjugation: & \alpha\,  \sim  {\rho^{\mp 1}}\, \alpha \, comb(\rho^{\pm 1}), & {\rm for} \ \alpha \in B_{2,n},\\
\end{array}
\]
\noindent, where $\rho$ denotes the loops $t,\, \tau$, or $T$, and $comb(\rho)$ denotes the combing of the loop $\rho$ through $B=\Sigma_1^{2p+1}$. In particular:
\[
\begin{array}{lcl}
comb(t)\ =\ T^{-p}\, t^{-1}\, T^{p+1} & , & comb(t^{-1})\ =\ T^{-p-1}\, t\, T^{p}\\
&&\\
comb(\tau)\ =\ T^{-p}\, t\, T^p & , & comb(\tau^{-1})\ =\ T^{-p}\, t^{-1}\, T^p\\
&&\\
comb(T^{\pm 1})\ =\ T^{\pm 1} &&\\
\end{array}
\]
\end{thm}

\section{Relating $KBSM(H_2)$ to $KBSM(T_{(2, 2p+1)}^c)$}\label{h2totc}

In this section we recall the knot theory of $H_2$ that is crucial for this paper and we establish the relation between $KBSM(H_2)$ and $KBSM(T_{(2, 2p+1)}^c)$, the Kauffman bracket skein module of the complement of $(2, 2p+1)$-torus knots.

\subsection{The knot theory of $H_2$}\label{h2kt}

We consider $H_2$ to be $S^3\backslash \{{\rm open\ tubular\ neighborhood\ of}\ I_2 \}$, where $I_2$ denotes the point-wise fixed identity braid on $2$ indefinitely extended strands meeting at the point at infinity. Thus, $H_2$ may be represented in $S^3$ by the braid $I_2$ (for an illustration see Figure~\ref{h2}).

\begin{figure}[ht]
\begin{center}
\includegraphics[width=2.1in]{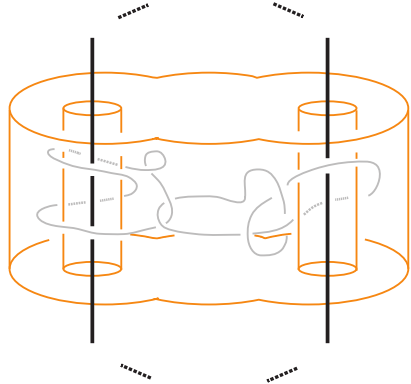}
\end{center}
\caption{A link in $H_2$.}
\label{h2}
\end{figure}

The main difference between a geometric mixed braid in $H_2$ and a geometric mixed braid in $T_{(2, 2p+1)}^c$ is the {\it closing} operation that is used to obtain mixed links from geometric mixed braids. In the case of $T_{(2, 2p+1)}^c$, the closing operation is defined in the usual sense (as in the case of classical braids in $S^3$), while in the case of $H_2$ the strands of the fixed part $I_2$ do not participate in the closure operation. In particular, the closure operation of a geometric mixed braid $I_2 \cup \beta$  in $H_2$ is realized by joining each pair of corresponding endpoints of the moving part by a vertical segment, either over or under the rest of the braid and according to the label attached to these endpoints. For more details the reader is referred to \cite{La1, OL, D3}.

\smallbreak

As explained in \cite{OL}, braid equivalence in $H_2$ is translated on the level of mixed braids in $S^3$ by a finite sequence of moves 1 and 2 in Theorem~\ref{algcco}, that is, stabilization moves and conjugation. Note also that $T$-loop conjugations in $H_2$, that is, $\alpha\ \sim\ T^{\pm 1}\ \alpha\ T^{\pm 1},\ \alpha \in B_{2, \infty}$, are allowed and that in \S~5 \cite{OL} it is shown how these conjugations can be translated in terms of relations (1) and (2) of Theorem~\ref{algcco}. Thus, the difference between braid equivalence in $H_2$ from braid equivalence in $T_{(2, 2p+1)}^c$, lies in the $t$ and $\tau$-combed loop conjugations, that is, the following moves correspond to braid equivalence in $T_{(2, 2p+1)}^c$ but not in $H_2$:

\begin{equation}\label{clceq}
\begin{array}{lclclcl}
\alpha & \sim & t^{-1}\, \alpha\, T^{-p}\, t^{-1}\, T^{p+1} & , & \alpha & \sim & t\, \alpha\, T^{-p-1}\, t\, T^p\\
&&&&&&\\
\alpha & \sim & \tau^{-1}\, \alpha\, T^{-p}\, t\, T^p & , & \alpha & \sim & \tau\, \alpha\, T^{-p}\, t^{-1}\, T^p\\
\end{array}
\end{equation}

Hence, in order to compute the Kauffman bracket skein module of the complement of $(2, 2p+1)$-torus knots, we need to consider the Kauffman bracket skein module of the genus 2 handlebody and study the effect of relations~(\ref{clceq}), i.e. combed loop conjugations, to elements in the basis of $KBSM(H_2)$. With a little abuse of notation, we denote this as follows:
\begin{equation}
KBSM(T_{(2, 2p+1)}^c)\ =\ \frac{KBSM(H_2)}{<combed\ loop\ conjugations>}.
\end{equation}

Our starting point is that a basis of $KBSM(H_2)$ spans $KBSM(T_{(2, 2p+1)}^c)$.

\subsection{The Kauffman bracket skein module of the genus 2 handlebody}\label{kbsmh2}

Recall now Theorem~\ref{tprz} for the basis of the Kauffman bracket skein module of $H_2$ and Figure~\ref{przt} for an illustration of elements in the basis. In order to present elements in the basis in open braid form, we need the following:

\begin{defn}\label{loopings}\rm
We define the following elements in $B_{2, n}$:
\begin{equation}\label{algloops}
t_i^{\prime}\ :=\ \sigma_i\ldots \sigma_1\, t\, \sigma^{-1}_1 \ldots \sigma^{-1}_i,\ \ \tau_k^{\prime}\ :=\ \sigma_k\ldots \sigma_1\, \tau\, \sigma^{-1}_1 \ldots \sigma^{-1}_k,\ \ T_j^{\prime}\ :=\ \sigma_j\ldots \sigma_1\, T\, \sigma^{-1}_1 \ldots \sigma^{-1}_j
\end{equation}
\end{defn}

Note that the $t_i^{\prime}$'s, the $\tau_k^{\prime}$'s and the $T_j^{\prime}$'s are conjugate but not commuting. In order to simplify the algebraic expressions obtained throughout this paper, we recall the following notation from \cite{D1}:

\begin{nt}\label{nt1}\rm
For $i, j\in \mathbb{N}$ such that $i<j$, we set:

\[
t^{a_{i,j}}_{i, j} \ := \ {t_i^{\prime}}^{a_i}\, {t_{i+1}^{\prime}}^{a_{i+1}}\, \ldots\, {t_j^{\prime}}^{a_j},\quad \tau^{b_{i, j}}_{i, j} \ := \ {\tau_i^{\prime}}^{b_i}\, {\tau_{i+1}^{\prime}}^{b_{i+1}}\, \ldots \, {\tau_j^{\prime}}^{b_j},\quad T^{c_{i, j}}_{i, j} \ := \ {T_i^{\prime}}^{c_i}\, {T_{i+1}^{\prime}}^{c_{i+1}}\, \ldots\, {T_j^{\prime}}^{c_j},
\]

\noindent where $a_k$, $b_k$ and $c_k \in \mathbb{N}, \forall k$, and

\begin{equation}
t_{i, j} \ := \ t_i^{\prime}\, t_{i+1}^{\prime}\, \ldots\, t_j^{\prime},\quad \tau_{i, j} \ := \ \tau_i^{\prime}\, \tau_{i+1}^{\prime}\, \ldots \, \tau_j^{\prime},\quad T_{i, j} \ := \ T_i^{\prime}\, T_{i+1}^{\prime}\, \ldots\, T_j^{\prime}.
\end{equation}

\noindent We also set:
$$t^{a_i, a_j}_{i, j}\, \tau^{b_{k, l}}_{k, l}\, T^{c_{m, n}}_{m, n}\, :=\, {t_i^{\prime}}^{a_i}\, \ldots \, {t_j^{\prime}}^{a_j}\, {\tau_{k}^{\prime}}^{b_{k}}\, \ldots \, {\tau_l^{\prime}}^{b_l}\, {T_{m}^{\prime}}^{c_{m}}\, \ldots\, {T_n^{\prime}}^{c_n},$$
\noindent where $i, j, k, l, m, n\in \mathbb{N}$ such that $i<j<k<l<m<n$.

\smallbreak

\noindent Finally, we note that ${t_i^{\prime}}^0\, {\tau_k^{\prime}}^0\, {T_j^{\prime}}^0$ represents the unknot.
\end{nt}

\smallbreak

We observe now that an element $x^iy^jz^k$ in the basis of $KBSM(H_2)$ described in Theorem~\ref{tprz} can be illustrated equivalently as a mixed link in $S^3$ so that we correspond the element $x^iy^jz^k$ to the minimal mixed braid representation in which we group the twists around each fixed strand. Figure~\ref{prbr} illustrates this correspondence. 

\begin{figure}[!ht]
\begin{center}
\includegraphics[width=4.1in]{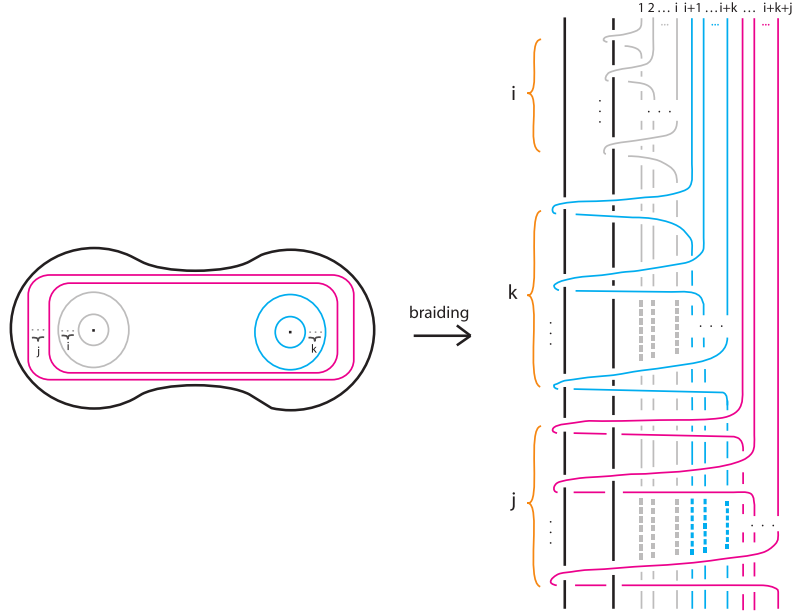}
\end{center}
\caption{Elements in the $B_{H_2}$ basis in braid form. }
\label{prbr}
\end{figure}

Then, the set

$$B_{H_2}\, :=\, \{(tt^{\prime}_1\ldots t^{\prime}_i)\, (\tau^{\prime}_{i+1}\tau^{\prime}_{i+2}\ldots \tau^{\prime}_{i+k})\, (T^{\prime}_{i+k+1}T^{\prime}_{i+k+2}\ldots T^{\prime}_{i+k+j})\}, \ (i, k, j)\in \mathbb{N}\times \mathbb{N}\times \mathbb{N}$$

\noindent is a basis of KBSM($H_2$) in braid form. Equivalently, using Notation~\ref{nt1} we have that

$$B_{H_2}\ :=\ \{t_{0, i}\, \tau_{i+1, i+k}\, T_{i+k+1, i+k+j} \}, \ (i, k, j)\in \mathbb{N}\times \mathbb{N}\times \mathbb{N}.$$

Note that when expressing an element in $B_{H_2}$ in open braid form, we obtain three different sets of strands: one set consists of the loopings $t^{\prime}_i$'s, the other set consists of the loopings $\tau^{\prime}_k$'s and the last one consists of the loopings $T^{\prime}_j$'s.

\bigbreak

Finally, it is worth mentioning that in \cite{D1}, two alternative bases for $KBSM(H_2)$ are presented in terms of braids (see Figure~\ref{basesall}). In particular, it is proved that the following sets forms bases for $KBSM(H_2)$:
\begin{equation}\label{nbasis1}
B^{\prime}_{H_2}\ =\ \{t^i\, {\tau^{\prime}_1}^k\, {T^{\prime}_2}^j,\ i, j, k \in \mathbb{N} \}.
\end{equation}
\begin{equation}\label{nbasis2}
\mathcal{B}_{H_2}\ =\ \{t^i\, \tau^k\, T^j,\ i, j, k \in \mathbb{N} \}.
\end{equation}

\begin{figure}[!ht]
\begin{center}
\includegraphics[width=4.8in]{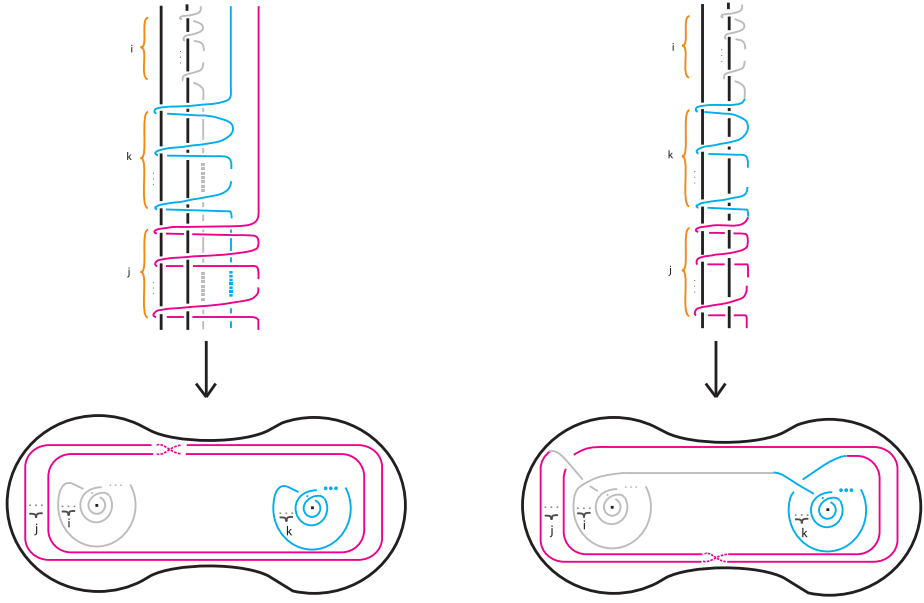}
\end{center}
\caption{Elements in the $B^{\prime}_{H_2}$ and $\mathcal{B}_{H_2}$ bases.}
\label{basesall}
\end{figure}

\noindent These bases are ``closer'' to the basis that we obtain for $KBSM(T_{(2, 2p+1)}^c)$ in \S~\ref{basTrc}.

\section{Spanning sets of $KBSM(T_{(2, 2p+1)}^c)$}\label{sstc}

In this section we study the effect of combed loop conjugations to elements in the basis of $KBSM(H_2)$. This is done in three steps. We first analyze relations~(\ref{clceq}) for the $t$ and $\tau$ loop generators and we conclude that the only effect that these relations have on elements in $B_{H_2}$, is that $t_i^{\prime}\, \sim \tau_i^{\prime}$, for all $i$. Thus, we drop the $\tau^{\prime}_i$'s and we have that monomials in the $t_i^{\prime}$'s and the $T^{\prime}_i$'s loop generators span $KBSM(T_{(2, 2p+1)}^c)$. We then deal with the $T$-generators and with the use of combed loop conjugations and skein relations, we obtain various spanning sets of $KBSM(T_{(2, 2p+1)}^c)$. Each spanning set we obtain consists of elements which are ``closer'' to the form of the elements of the (new) basis of $KBSM(T_{(2, 2p+1)}^c)$, that is presented in section \S~\ref{basTrc}.

\subsection{Dealing with the $t$ and the $\tau$ loop generators}

As mentioned in the introduction of this section, we first deal with the $t$'s and $\tau$'s loop generators by applying combed loop conjugations. In particular, we show that the $\tau_i$'s loop generators may be replaced by the $t_i$'s loopings, since both are meridians in $T_{(2, 2p+1)}^c$. We prove this fact using the technique of combing. 

\smallbreak

Note that we will be using the symbol ``\ $\simeq$\ '' when braid equivalence is involved and the symbol ``\ $\widehat{=}$\ '' when isotopy is involved in the closure (conjugation). We will pass from braids to links and vice-versa in $T_{(2, 2p+1)}^c$ in order to take advantage of the fact that the looping generators are conjugates. Finally, when both braid equivalence and isotopy in the closure is considered, we will use the symbol ``\ $\widehat{\simeq}$\ ''. 

\bigbreak

\begin{lemma}\label{taulem}
The following holds in the $KBSM(T_{(2, 2p+1)}^c)$.
\[\tau_i^{\prime}\ \widehat{\simeq}\ t_i^{\prime},\, \forall i.
\]
\end{lemma}

\begin{proof}
\[
\tau_i^{\prime}\ \sim\ {\tau_i^{-1}}^{\prime}\, \tau_i^{\prime}\, comb\left(\tau_i^{\prime}\right)\ \sim\ comb\left(\tau_i^{\prime}\right)\ \overset{Rel.~(\ref{clceq})}{=}\ {T_i^{\prime}}^{-p}\, t_i^{\prime}\, {T_i^{\prime}}^p\ \widehat{\simeq}\ t_i^{\prime}.
\]
\end{proof}

Note that Lemma~\ref{taulem} can be applied to all $\tau_i^{\prime}$'s on elements in $B_{H_2}$ since the looping generators are conjugates. Thus, we may first apply Lemma~\ref{taulem} on each looping $\tau_i^{\prime}$, then consider the closure of the resulting mixed braid, treat each component separately, and finally consider the result in open braid form. This leads to the following result:

\begin{prop}\label{tautot}
The set $\mathcal{T}_1\ =\ \{t_{0, n}^{\prime}\, T_{n+1, n+m}^{\prime}\ |\ n, m\in \mathbb{N}\cup\{0\} \}$ spans $KBSM(T_{(2, 2p+1)}^c)$.
\end{prop}

\begin{remark}\rm
Obviously, one may consider elements of the form
\[
\tau_{0, n}^{\prime}\, T_{n+1, n+m}^{\prime}
\]
\noindent to generate $KBSM(T_{(2, 2p+1)}^c)$.
\end{remark}

We now prove that combed loop conjugation on an element $w$ in $B_{H_2}$ may only result to an element in $\mathcal{T}_1$.

\begin{lemma}\label{clclem}
Combed loop conjugations on any looping generator, either has the effect of Lemma~\ref{taulem}, or has no effect at all. 
\end{lemma}

\begin{proof}
We only prove Lemma~\ref{clclem} for some cases. All other cases follow similarly. Note that we will be using the fact that $\tau_i^{\prime}\ :=\ {T_i^{\prime}}\, {t_i^{\prime}}^{-1}$ and ${\tau_i^{\prime}}^{-1}\ :=\ t_i^{\prime}\, {T_i^{\prime}}^{-1}$.

\bigbreak

\begin{itemize}
\item[i.]  $t_i^{\prime} \ \sim \ {t_i^{\prime}}^{-1}\, t_i^{\prime}\, comb\left(t_i^{\prime}\right) \ \sim \ comb\left(t_i^{\prime}\right) \ \overset{Rel.~(\ref{clceq})}{=} \ {T_i^{\prime}}^{-p}\, {t_i^{\prime}}^{-1}\, {T_i^{\prime}}^{p+1} \ \widehat{\simeq}\ {t_i^{\prime}}^{-1}\, {T_i^{\prime}}\ \widehat{\sim}\ {T_i^{\prime}}\, {t_i^{\prime}}^{-1} \ :=\ \tau_i^{\prime}$.
\bigbreak
\item[ii.] $t_i^{\prime} \ \sim \ {\tau_i^{\prime}}\, t_i^{\prime}\, comb\left({\tau_i^{\prime}}^{-1}\right) \ \sim \ \tau_i^{\prime}\,  t_i^{\prime}\,  {T_i^{\prime}}^{-p}\, {t_i^{\prime}}^{-1}\, {T_i^{\prime}}^p\ \simeq\ {T_i^{\prime}}^{1-p}\, {t_i^{\prime}}^{-1}\, {T_i^{\prime}}^p\ \widehat{=}\ {T_i^{\prime}}\, {t_i^{\prime}}^{-1} \ =\ \tau_i^{\prime}$.
\bigbreak
\item[iii.] $t_i^{\prime} \ \sim \ {T_i^{\prime}}^{\pm 1}\, t_i^{\prime}\, comb\left({T_i^{\prime}}^{\mp 1}\right)\ \overset{Rel.~(\ref{clceq})}{=}\ {T_i^{\prime}}^{\pm 1}\, t_i^{\prime}\, {T_i^{\prime}}^{\mp 1}\ \widehat{\simeq}\ t_i^{\prime}$.
\bigbreak
\item[iv.] $T_i^{\prime}\ \sim\ t_i^{\prime}\, T_i^{\prime}\, comb\left({t_i^{\prime}}^{-1}\right)\ \overset{Rel.~(\ref{clceq})}{=}\ t_i^{\prime}\, T_i^{\prime}\, {T_i^{\prime}}^{-p-1}\, t_i^{\prime}\, {T_i^{\prime}}^p\ \sim\ t_i^{\prime}\, {T_i^{\prime}}^{-p}\, t_i^{\prime}\, {T_i^{\prime}}^p\ =\ {\tau_i^{\prime}}^{-1}\, {T_i^{\prime}}^{-p+1}\, t_i^{\prime}\, {T_i^{\prime}}^p $\\
$\sim\ {T_{i}^{\prime}}^{-p+1} t_i^{\prime}\, {T_i^{\prime}}^p\, comb\left(\tau_i^{\prime}\right)\ \overset{Rel.~(\ref{clceq})}{=}\  {T_i^{\prime}}^{-p+1}\, t_i^{\prime}\, {T_i^{\prime}}^p\, {T_i^{\prime}}^{-p}\, {t_i^{\prime}}^{-1}\, {T_i^{\prime}}^{p}\ \widehat{=}\ T_i^{\prime}$.
\end{itemize}

\end{proof}

\subsection{Dealing with the $T$-loop generators}

We now deal with the $T$-loop generators on elements in $\mathcal{T}_1$. Note that from Lemma~\ref{clclem} we have that combed loop conjugations cannot reduce the number of $T_i$'s in monomials in $\mathcal{T}_1$, but relations~(\ref{clceq}) suggest that we can reduce them if we first manage to ``group'' them somehow (see Lemma~\ref{grT} below). We first introduce the notion of the {\it index} of a monomial in the $t_i^{\prime}$'s and $T_i^{\prime}$'s.

\begin{defn}\rm
Let $w$ be a monomial in the $t_i^{\prime}$'s and $T_i^{\prime}$'s. We define $in(w)$ to be the highest index in $w$.
\end{defn}

\begin{lemma}\label{grT}
Let $w\in \mathcal{T}_1$ such that $in(w)=n$ or $n+1$. Then:
\[
w\, T_{n+1}^{\prime}\, T_{n+2}^{\prime}\ \widehat{\simeq} \ -A^2\, w\, {T_{n+1}^{\prime}}^2\, +\, A\, (-A^2-A^{-2})\, w.
\]
\end{lemma}

\begin{proof}
The proof of Lemma~\ref{grT} is illustrated in Figure~\ref{prooflem} for the case $in(w)=n$. The case $in(w)=n+1$ follows similarly. Note that in Figure~\ref{prooflem} we denote the knotting of the fixed part by $\Sigma\, :=\, \Sigma_1^{2p+1}$.
\end{proof}

\begin{figure}[!ht]
\begin{center}
\includegraphics[width=5.3in]{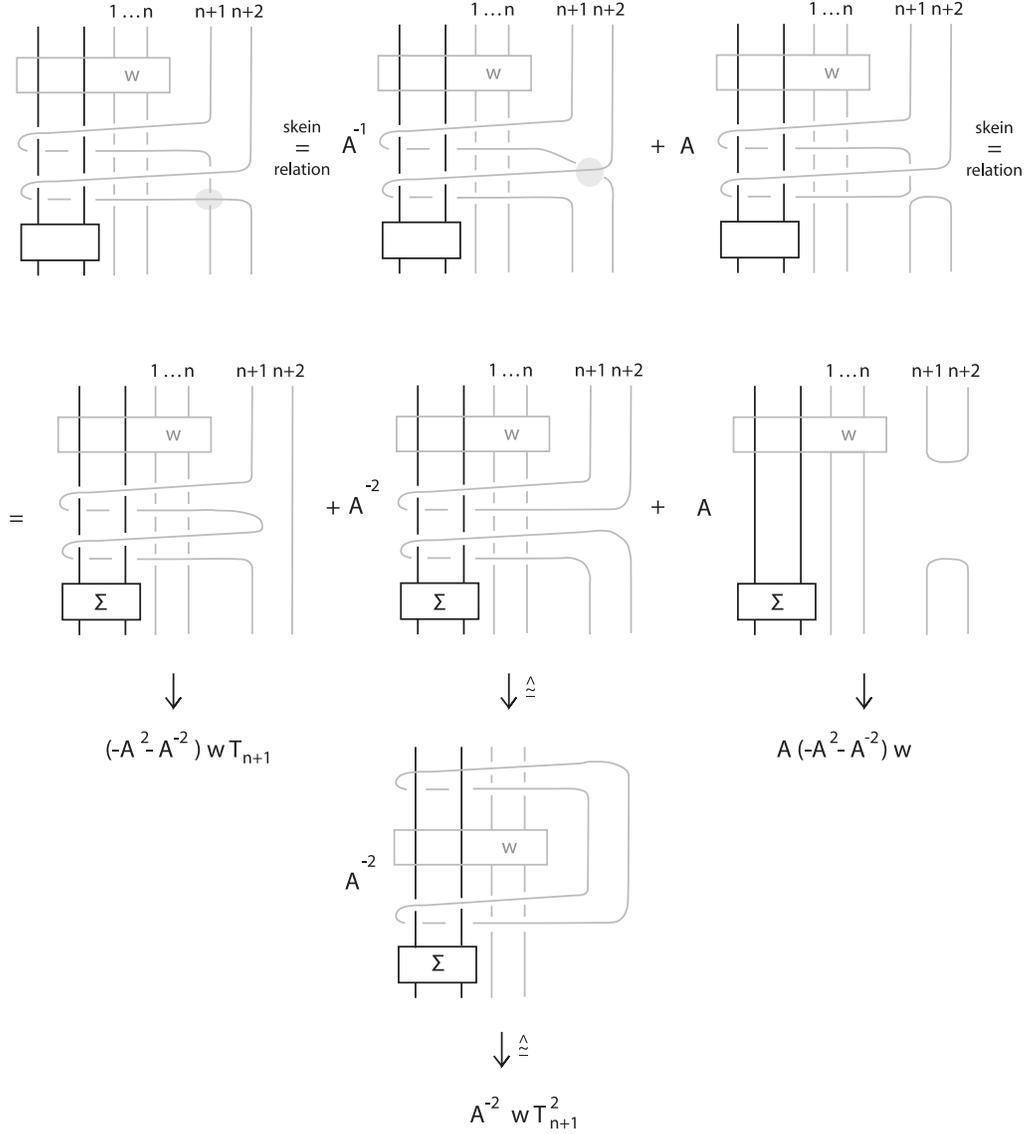}
\end{center}
\caption{The proof of Lemma~\ref{grT}.}
\label{prooflem}
\end{figure}

\noindent Note that the relations in Lemma~\ref{grT} can be generalized for arbitrary exponents (following the same steps in the proof) as follows:  

\begin{lemma}\label{grTkc}
Let $w$ be a monomial in the $t_i^{\prime}$'s and $T_i^{\prime}$'s such that $in(w)=n+1$ or $in(w)=n$. Then, the following holds in $KBSM(T_{(2, 2p+1)}^c)$:
\[
w\, {T_{n}^{\prime}}^k\, {T_{n+2}}^{\prime} \ \widehat{\simeq} \ -A^2\, w\, {T_{n+1}^{\prime}}^{k+1}\ +\ A\, (-A^2-A^{-2})\, w\, {T_{n+1}^{\prime}}^{k-1}.
\]
\end{lemma}

\begin{proof}
We prove Lemma~\ref{grTkc} by induction on $k\in \mathbb{N}$. The basis of induction is Lemma~\ref{grT}. Assume now that Lemma~\ref{grTkc} is true for $k$. Then, for $k+1$ we have:

\[
\begin{array}{lcl}
w\, {T_{n}^{\prime}}^{k+1}\, {T_{n+2}}^{\prime} & \widehat{\simeq} & w\, {T_{n}^{\prime}}\, \underline{{T_{n}^{\prime}}^{k}\, {T_{n+2}}^{\prime}}\ \underset{Ind.\, Step}{\widehat{\simeq}}\\
&&\\
& \widehat{\simeq} & -A^2\, w\, T_{n+1}^{\prime}\, {T_{n+1}^{\prime}}^{k+1}\ +\ A\, (-A^2-A^{-2})\, w\, T_{n+1}^{\prime}\, {T_{n+1}^{\prime}}^{k-1}\ =\\
&&\\
& = & -A^2\, w\, {T_{n+1}^{\prime}}^{k+1}\ +\ A\, (-A^2-A^{-2})\, w\, {T_{n+1}^{\prime}}^{k-1}.
\end{array}
\]
\end{proof}

Using Lemmas~\ref{grT} and \ref{grTkc}, we observe that we may ``group'' the $T$-generators, $T_{n+1, n+k}^{\prime}$, together to one $T$-generator, $T_{n+1}$, raised to an arbitrary exponent. In what follows we drop the coefficients and we have the following result:

\begin{prop}\label{groupT}
Let $w$ be a monomial in the $\mathcal{T}_1$. Then, the following holds in $KBSM(T_{(2, 2p+1)}^c)$:
\[
t_{0, n}^{\prime}\, {T_{n+1, n+k}^{\prime}} \ \widehat{\simeq} \ \begin{cases} \underset{i=0}{\overset{\frac{k}{2}}{\sum}}\, t_{0, n}^{\prime}\, {T_{n+1}^{\prime}}^{2i} &,\, k-even\\
\underset{i=0}{\overset{\frac{k+1}{2}}{\sum}}\, t_{0, n}^{\prime}\, {T_{n+1}^{\prime}}^{2i+1} &,\, k-odd.
  \end{cases}
\]
\end{prop}

\begin{proof}
We prove Proposition~\ref{groupT} by induction on $k\in \mathbb{N}$. The base of induction is Lemma~\ref{grTkc} for $k=1$. Assume now that the relation holds for $k$-odd (the case $k$-even follows similarly). Then, for $k+2$ we have:

\[
\begin{array}{lcl}
t_{0, n}^{\prime}\, {T_{n+1, n+k+2}^{\prime}} & \widehat{\simeq} & t_{0, n}^{\prime}\, {T_{n+1, n+2}^{\prime}}\, \underline{{T_{n+3, n+k+2}^{\prime}}}\ \underset{Ind.\, Step}{\widehat{\simeq}}\\
&&\\
& \widehat{\simeq} &  \underset{i=0}{\overset{\frac{k+1}{2}}{\sum}}\, t_{0, n}^{\prime}\, {T_{n+1, n+2}^{\prime}}\, {T_{n+3}^{\prime}}^{2i+1} \ \widehat{=}\ \underset{i=0}{\overset{\frac{k+1}{2}}{\sum}}\, t_{0, n}^{\prime}\, {T_{n+1}^{\prime}}\, {T_{n+2}^{\prime}}^{2i+1}\, T_{n+3}^{\prime} \ \underset{Lemma~\ref{grTkc}}{\widehat{\simeq}}\\
&&\\
&  \widehat{\simeq} & \underset{i=0}{\overset{\frac{k+1}{2}}{\sum}}\, t_{0, n}^{\prime}\, {T_{n+1}^{\prime}}\, {T_{n+2}^{\prime}}^{2i+2}\ +\ \underset{i=0}{\overset{\frac{k+1}{2}}{\sum}}\, t_{0, n}^{\prime}\, {T_{n+1}^{\prime}}\, {T_{n+2}^{\prime}}^{2i}\ \widehat{=}\\
&&\\
& \widehat{=} & \underset{i=0}{\overset{\frac{k+1}{2}}{\sum}}\, t_{0, n}^{\prime}\, {T_{n+1}^{\prime}}^{2i+2}\, \, {T_{n+2}^{\prime}}\ +\ \underset{i=0}{\overset{\frac{k+1}{2}}{\sum}}\, t_{0, n}^{\prime}\, {T_{n+1}^{\prime}}^{2i}\, \, {T_{n+2}^{\prime}}\ \underset{Lemma~\ref{grTkc}}{\widehat{\simeq}}\\
&&\\
& \widehat{\simeq} & \underset{i=0}{\overset{\frac{k+1}{2}}{\sum}}\, t_{0, n}^{\prime}\, {T_{n+1}^{\prime}}^{2i+3}\ +\ \underset{i=0}{\overset{\frac{k+1}{2}}{\sum}}\, t_{0, n}^{\prime}\, {T_{n+1}^{\prime}}^{2i+1}\ +\ \underset{i=1}{\overset{\frac{k+1}{2}}{\sum}}\, t_{0, n}^{\prime}\, {T_{n+1}^{\prime}}^{2i+-1} \ \overset{j=i+1}{=}\\
&&\\
& = & \underset{j=0}{\overset{\frac{k+3}{2}}{\sum}}\, t_{0, n}^{\prime}\, {T_{n+1}^{\prime}}^{2j+1}.
\end{array}
\]
\end{proof}

We now present the following fundamental result based on combing:

\begin{lemma}\label{l1}
Let $w$ be a monomial in the $t_i^{\prime}$'s and $T_i^{\prime}$'s such that $in(w)=n$ and let $k\in \mathbb{N}\cup \{0\}$. Then, the following holds in $KBSM(T_{(2, 2p+1)}^c)$:
\bigbreak

\[
\begin{array}{rrcl}
i. & w\, {T_{n+1}^{\prime}}^{p+1} & \widehat{\simeq} & w\, {t_{n+1}^{\prime}}^{2}\, {T_{n+1}^{\prime}}^p\\
&&&\\
ii. & w\, {t_{n+1}^{\prime}}^k\, {T_{n+1}^{\prime}}^{p+1} & \widehat{\simeq} & w\, {t_{n+1}^{\prime}}^{k+2}\, {T_{n+1}^{\prime}}^p.
\end{array}
\]

\end{lemma}

\begin{proof}
The relations follow as a result of combed loop conjugation (recall Rel.~(\ref{clceq})). Indeed, for relation (ii) we have the following:
\[
\begin{array}{lcl}
w\, {t_{n+1}^{\prime}}^k\, {T_{n+1}^{\prime}}^{p+1} & \sim & w\, {t_{n+1}^{\prime}}^k\, {t_{n+1}^{\prime}}\, {T_{n+1}^{\prime}}^p\, \left({T_{n+1}^{\prime}}^{-p}\, {t_{n+1}^{\prime}}^{-1}\, {T_{n+1}^{\prime}}^{p+1} \right)\ =\\
&&\\
&= & w\, {t_{n+1}^{\prime}}^{k+1}\, {T_{n+1}^{\prime}}^p\, comb(t_{n+1}^{\prime})\ \widehat{\simeq}\ w\, {t_{n+1}^{\prime}}^{k+2}\, {T_{n+1}^{\prime}}^p.
\end{array}
\]
\noindent Relation (i) follows similarly.
\end{proof}



An immediate result of Lemmas~\ref{grT}, \ref{grTkc}, \ref{l1} and Proposition~\ref{groupT} is the following:

\begin{lemma}\label{lsq}
For $m>p+1$, the following holds in $KBSM(T_{(2, 2p+1)}^c)$:
\begin{equation}\label{Tsq}
t_{0, n}^{\prime}\, T_{n+1, n+m}^{\prime}\ \widehat{\simeq}\ \underset{i=0}{\overset{p}{\sum}}\, t_{0, n}^{\prime}\, {t_{n+1}^{\prime}}^{k+i}\, {T_{n+1}^{\prime}}^i,
\end{equation}
\noindent where $k_i\in \mathbb{N}$ even.
\end{lemma}


\begin{cor}
The set $\mathcal{T}_2$ that consists of elements in the form of Equation~(\ref{Tsq}) is a spanning set of $KBSM(T_{(2, 2p+1)}^c)$.
\end{cor}

Consider now an element $w$ in $\mathcal{T}_2$, $t_{0, n}^{\prime}\, T_{n+1, n+m}^{\prime}$ such that $m=p+k$, for $k\in \mathbb{N}$ and $k<p$. Then, write $w$ in the following form 

\[
t_{0, n}^{\prime}\, T_{n+1, n+p+k}^{\prime}\ =\ t_{0, n}^{\prime}\, \underline{T_{n+1, n+k}^{\prime}\, T_{n+k+1, n+k+p}^{\prime}}
\]

\noindent and apply Proposition~\ref{groupT} on the underlined expression to get:

\[
\underset{j=0}{\overset{p+1}{\sum}}\, t_{0, n}^{\prime}\, T_{n+1, n+k}^{\prime}\, {T_{n+k+1}^{\prime}}^{j}.
\]

\noindent For $j=p+1$, we apply Lemma~\ref{l1} in order to reduce the exponent $p+1$ to $p$ and we then group the $T_i^{\prime}$'s again in order to increase the exponent to $p+1$.

\bigbreak

Continuing in that way, and if we let ${t_{k, k+p}^{\prime}}^{\lambda_{k, k+p}}\ :=\  {t_{k}^{\prime}}^{\lambda_{k}}\ldots {t_{k+p}^{\prime}}^{\lambda_{k+p}}$, for $k, p\in \mathbb{N}$, we will eventually obtain a sum of elements of the following forms:

\begin{equation}\label{bel}
\begin{array}{cll}
(a)-form: & {t_{0, n}^{\prime}}^{\lambda_{0, n}}\\
&&\\
(b)-form: & {t_{0, n}^{\prime}}^{\mu_{0, n}}\, {T_{n+1}^{\prime}}^{j}\\
&&\\
(c)-form: & {t_{0, n}^{\prime}}^{\phi_{0, n}}\, {T_{n}^{\prime}}^{j},
\end{array}
\end{equation}

\noindent where $n, \lambda_i, \mu_i, \phi_i\in \mathbb{N}$ for all $i$a and $j\in \{0, 1, \ldots, p\}$. Thus, we have reduced the number of the $T$-loop generators needed in order to generate $KBSM(T_{2, 2p+1}^c)$ and we have proved the following result:

\begin{prop}
The set
\[
\mathcal{T}_3\ =\ \left\{{t_{0, n}^{\prime}}^{k_{0, n}}\, {T_{n+j}^{\prime}}^{v},\ |\ {\rm where}\ n, k_i\in \mathbb{N}\ {\rm for\ all}\ i\ {\rm and}\ j, v\in \{0, 1\}\right\}.
\]
\noindent spans $KBSM(T_{2, 2p+1}^c)$.
\end{prop}

Obviously, the set $\mathcal{T}_3$ consists of elements in the (a), (b) and (c)-forms (Eq.~(\ref{bel})). In the next section we treat elements in the three forms of Eq.~(\ref{bel}) in order to reach at a basis of $KBSM(T_{2, 2p+1}^c)$.

\section{The basis $\mathcal{B}_{T_{2, 2p+1}^c}$ of $KBSM(T_{2, 2p+1}^c$}\label{basTrc}

In this section we obtain the basis $\mathcal{B}_{T_{2, 2p+1}^c}$ of $KBSM(T_{2, 2p+1}^c)$. Our starting point is the spanning set $\mathcal{T}_3$ and an ordering relation defined in $\mathcal{T}_3$. Note that elements in $\mathcal{T}_3$ consist of monomials in the $t_i^{\prime}$'s without gaps in the indices and with arbitrary exponents in $\mathbb{N}$, followed maybe by a $T$-generator that either has the same index as the highest index of the monomial of the $t_i^{\prime}$'s, say $m$, or the index of the $T$-generator is $m+1$.

\subsection{An ordering relation on $\mathcal{T}_3$}

Before we proceed with the ordering relation defined on the set $\mathcal{T}_3$, we introduce the following notions:

\begin{defn}\rm
The $t$-index of a monomial $w$ in $\mathcal{T}_3$, denoted $ind_t(w)$, is defined to be the highest index of the $t_i^{\prime}$'s in $w$. Similarly, the $T$-index, $ind_{T}(w)$, is defined to be the highest index of the $T_i^{\prime}$'s in $w$. We also define $exp(w)$ to be the sum of all exponents in $w$.
\end{defn}

We will also write $t_i$ instead of $t_i^{\prime}$ and $T_i$ instead of $T_i^{\prime}$ from now on, in order to simplify the relations. We are now in position to define an ordering relation on $\mathcal{T}_3$.

\begin{defn}\label{order}\rm
Let $n, m, k_i, r_i \in \mathbb{N}$ for all $i$ and let:
\[
w\, =\, t_{0, n}^{k_{0, n}}\, T_{n+j}^{v_1}\ {\rm and}\ u\, =\, t_{0, m}^{r_{0, m}}\, T_{m+j}^{v_2},\ {\rm where}\ j\in \{0, 1\}\ {\rm and}\ v_i\in \{0, \ldots, p\}\ {\rm for}\ i=1, 2.
\]

\smallbreak

\noindent Then, we define the following ordering relation:

\smallbreak

\begin{itemize}
\item[(A)] If $exp(w)\,  <\, exp(u)$, then $w<u$.

\vspace{.1in}

\item[(B)] If $exp(w)\, =\, exp(u)$, then:

\vspace{.1in}

\noindent  (i) if $ind_T(w)\, <\, ind_T(u)$, then $w<u$,

\vspace{.1in}

\noindent  (ii) if $ind_T(w)\, =\, ind_T(u)$, then:

\vspace{.1in}

\noindent \ \ \ \ (a) if $ind_{t}(w)\, <\, ind_{t}(u)$, then $w<u$,

\vspace{.1in}

\noindent \ \ \ \ (b) if $ind_{t}(w)\, =\, ind_{t}(u)$ and $v_2<v_1$, then $w<u$.

\vspace{.1in}

\noindent \ \ \ \  (c) if $ind_{t}(w)\, =\, ind_{t}(u)$ and $v_2=v_1$ then:

\vspace{.1in}

\noindent \ \ \ \ \ \ \ ($\bullet$) if $k_n=r_n, \ldots, k_{n-i}=r_{n-1}, k_{n-i-1}<r_{n-i-1}$, then $w<u$.

\vspace{.1in}

\noindent \ \ \ \ \ \ \ ($\bullet$) if $k_{i}=r_{i}, \forall i$, then $w=u$. 
\end{itemize}
\end{defn}

\begin{prop}\label{totord}
The set $\mathcal{T}_3$ equipped with the ordering relation of Definition~\ref{order}, is a totally ordered and a well-ordered set.
\end{prop}

\begin{proof}
In order to show that the set $\mathcal{T}_3$ is a totally ordered set when equipped with the ordering of Definition~\ref{order}, we need to show that the ordering relation is antisymmetric, transitive and total. We only show that the ordering relation is transitive. Antisymmetric property follows similarly. Totality follows from Definition~\ref{order} since all possible cases have been considered. Let:

\[
w \ = \ t_{0, n}^{k_{0, n}}\, T_{n+j_1}^{\mu_1},\ \ \ u \ = \ t_{0, m}^{r_{0, m}}\, T_{m+j_2}^{\mu_2},\ \ \ v \ = \ t_{0, q}^{l_{0, q}}\, T_{q+j_3}^{\mu_3},
\]

\noindent where $j_1, j_2, j_3\in \{0, 1\}$, $\mu_1, \mu_2, \mu_3\in \{0, \ldots, p\}$ and such that $w<u$ and $u<v$. Then, one of the following holds:

\bigbreak

\begin{itemize}
\item[a.] Either $exp(w)<exp(u)$, and since $u<v$, we have $exp(u)\leq exp(v)$ and thus, $u<v$.
\bigbreak
\item[b.] Either $exp(w)=exp(u)$ and $ind_T(w)<ind_T(u)$. Then, since $u<v$ we have that either $exp(u)<exp(v)$ (same as in case (a)), or $exp(u)=exp(v)$ and $ind_T(u)\leq ind_T(v)$. Thus, $ind_T(w)<ind_T(v)$ and so we conclude that $w<v$.
\bigbreak
\item[c.] Either $exp(w)=exp(u)$, $ind_T(w)=ind_T(u)$ and $ind_t(w)<ind_t(u)$. Then, since $u<v$, we have that either:
\bigbreak
\begin{itemize}
\item[$\bullet$] $exp(u)<exp(v)$, same as in case (a), or
\smallbreak
\item[$\bullet$] $exp(u)=exp(v)$ and $ind_T(u)<ind_T(v)$, same as in case (b), or
\smallbreak
\item[$\bullet$] $exp(u)=exp(v)$, $ind_T(u)=ind_T(v)$ and $ind_t(u)<ind_t(v)$. Thus, $w<v$.
\end{itemize}
\bigbreak
\item[d.] Either $exp(w)=exp(u)$, $ind_T(w)=ind_T(u)$, $ind_t(w)=ind_t(u)$ and $\mu_1<\mu_2$. Then, since $u<v$, we have that either:
\begin{itemize}
\item[(i)] $exp(u)<exp(v)$, same as in case (a), or
\smallbreak
\item[(ii)] $exp(u)=exp(v)$ and $ind_T(u)<ind_T(v)$, same as in case (b), or
\smallbreak
\item[(iii)] $exp(u)=exp(v)$, $ind_T(u)=ind_T(v)$ and $ind_t(u)<ind_t(v)$, same as in case (c), or
\smallbreak
\item[(iv)] $exp(u)=exp(v)$, $ind_T(u)=ind_T(v)$, $ind_t(u)<ind_t(v)$ and $\mu_2<\mu_3$. Then: $\mu_1<\mu_3$ and thus, $w<v$.
\end{itemize}
\bigbreak
\item[e.] Either $exp(w)=exp(u)$, $ind_T(w)=ind_T(u)$, $ind_t(w)=ind_t(u)$, $\mu_1=\mu_2$ and $k_n=r_n, \ldots, k_{n-i}=r_{n-1}, k_{n-i-1}<r_{n-i-1}$. Then, since $u<v$, we have that either: 
\begin{itemize}
\item[(i)] $exp(u)<exp(v)$, same as in case (a), or
\smallbreak
\item[(ii)] $exp(u)=exp(v)$ and $ind_T(u)<ind_T(v)$, same as in case (b), or
\smallbreak
\item[(iii)] $exp(u)=exp(v)$, $ind_T(u)=ind_T(v)$ and $ind_t(u)<ind_t(v)$, same as in case (c), or
\smallbreak
\item[(iv)] $exp(w)=exp(u)$, $ind_T(w)=ind_T(u)$, $ind_t(w)=ind_t(u)$ and $\mu_2<\mu_3$, same as in case (d), or
\smallbreak
\item[(v)] $exp(u)=exp(v)$, $ind_T(u)=ind_T(v)$, $ind_t(u)<ind_t(v)$, $\mu_2=\mu_3$ and $r_n=p_n, \ldots, r_{n-y}=q_{n-y}, r_{n-y-1}<q_{n-y-1}$. Then:
\smallbreak
\smallbreak
\begin{itemize}
\item[(-)] if $y=i$, then $k_{n-i-1}<r_{n-i-1}<q_{n-i-1}$ and thus $w<v$.
\smallbreak
\item[(=)] if $y>i$, then $k_{n-y-1}=r_{n-y-1}<q_{n-y-1}$ and thus $w<v$.
\end{itemize}

\end{itemize}
\end{itemize}

\bigbreak

So, we conclude that the ordering relation is transitive. Moreover, the element $t^{0}T^0$ that represents the unknot, is the minimum element of $\mathcal{T}_3$ and so $\mathcal{T}_3$ is a well-ordered set.
\end{proof}

\begin{remark}\rm
Proposition~\ref{totord} also holds for the set $\mathcal{T}_2$ since $\mathcal{T}_2$ is a proper subset of $\mathcal{T}_3$.
\end{remark}

\subsection{The set $\mathcal{B}_{T_{(2, 2p+1)}^c}$ as a spanning set of $KBSM(T_{(2, 2p+1)}^c)$}

In this subsection we reduce elements in $\mathcal{T}_3$, that is, elements of the (a), (b), (c)-forms of Eq.~(\ref{bel}), to elements in the set
\[
\mathcal{B}_{Tr^c}\ =\ \left\{ t^n\, T^v\ | \ n\in \mathbb{N}\cup\{0\},\ v\in \{0, \ldots, p\}\right\}.
\]

We deal with elements in the (a)-form of Eq.~(\ref{bel}) first.

\begin{lemma}\label{aform}
Let $n, k_i\in \mathbb{N}$ for all $i$. Then, the following holds in $KBSM(T_{(2, 2p+1)}^c)$:
\[
t_{0, n}^{k_{0, n}}\ \widehat{\simeq}\ \underset{i=0}{\overset{k}{\sum}}\, c_i\cdot t^i,\ {\rm where}\ k=\underset{i=0}{\overset{n}{\sum}}\, k_i,\ {\rm and}\ c_i\ {\rm coefficients\ for\ all}\ i.
\]
\end{lemma}

\begin{proof}
We prove Lemma~\ref{aform} by strong induction on the order of $t_{0, n}^{k_{0, n}} \in \mathcal{T}_3$.

\smallbreak
The base of induction is $tt_1\, \widehat{\simeq}\ c_1\, t^2\, +\, c_2\, t^0$, which is illustrated in Figure~\ref{reducett} by omitting the coefficients. 
\begin{figure}[!ht]
\begin{center}
\includegraphics[width=4.4in]{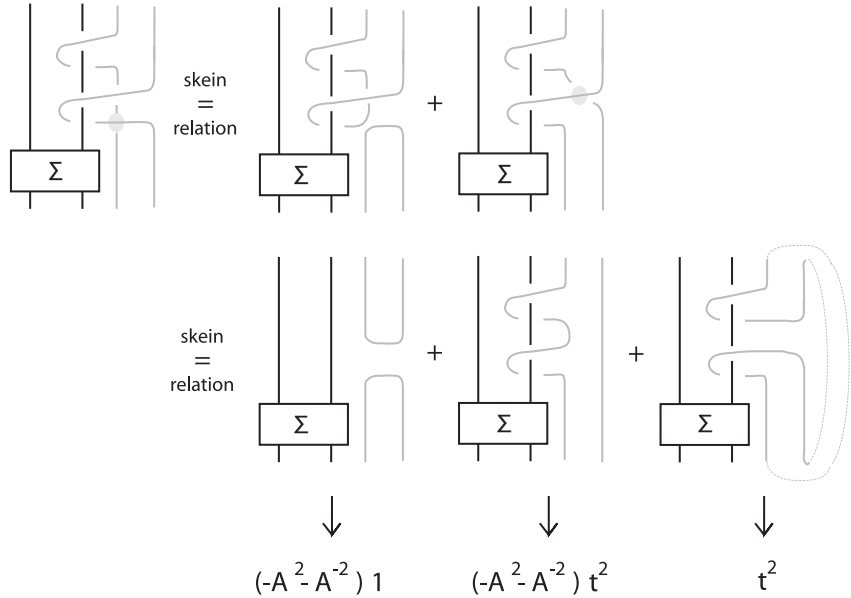}
\end{center}
\caption{The base of induction of Lemma~\ref{aform}.}
\label{reducett}
\end{figure}

Assume now that Lemma~\ref{aform} is true for all $w\, <\, t_{0, n}^{k_{0, n}} \in \mathcal{T}_3$. Then, for $t_{0, n}^{k_{0, n}}$, we repeat the steps shown in Figure~\ref{reducett} for the two $t$-loopings with the highest index and we have that:
\[
t_{0, n}^{k_{0, n}}\ =\ t_{0, n-2}^{k_{0, n-2}}\, \underline{t_{n-1}^{k_{n-1}}\, t_{n}^{k_{n}}}\ \overset{skein}{\underset{relation}{=}}\ t_{0, n-2}^{k_{0, n-2}}\, t_{n-1}^{k_{n-1}+k_{n}-2}\ +\ t_{0, n-2}^{k_{0, n-2}}\, t_{n-1}^{k_{n-1}-1}\, t_n^{k_{n}-1}\ +\ t_{0, n-2}^{k_{0, n-2}}\, t_{n-1}^{k_{n-1}+k_{n}}.
\]

According to Definition~\ref{order}, all monomials on the right hand side of the relation above are of less order than the initial monomial $t_{0, n}^{k_{0, n}}$, since the first and the third terms have index $(n-1)$, which is less than the index of $t_{0, n}^{k_{0, n}}$, and the exponent of $t_n$ in the second term is $k_{n}-1$, which is less than the exponent, $k_n$, of $t_n$ in $t_{0, n}^{k_{0, n}}$. Thus, the result follows from the induction hypothesis.
\end{proof}

We now treat elements in the (b)-form of Eq.~(\ref{bel}).

\begin{lemma}\label{bform}
Let $n, k_i\in \mathbb{N}$ for all $i$. Then, the following holds in $KBSM(T_{(2, 2p+1)}^c)$:
\[
t_{0, n}^{k_{0, n}}\, T_{n+1}^{v}\ \widehat{\simeq}\ \underset{i=0}{\overset{k}{\sum}}\, c_i\cdot t^i\, T_1^{v},\ {\rm where}\ k=\underset{i=0}{\overset{n}{\sum}}\, k_i,\ v\in \{0, \ldots, p\}\ {\rm and}\ c_i\ {\rm coefficients\ for\ all}\ i.
\]
\end{lemma}

\begin{proof}
The proof is similar to that of Lemma~\ref{aform}, by observing that 
\[
t_{0, n}^{k_{0, n}}\, T_{n+1}^{v}\ \widehat{=}\ T^{v}\, t_{1, n+1}^{k_{0, n}}
\]
\noindent and that a $T$-loop in the beginning of a $t$-monomial will not affect the steps of the proof of Lemma~\ref{aform}.
\end{proof}

Finally, we manage elements in the (c)-form of Eq.~(\ref{bel}) as follows:

\begin{lemma}\label{cform}
Let $n, k_i\in \mathbb{N}$ for all $i$. Then, the following holds in $KBSM(T_{(2, 2p+1)}^c)$:
\[
t_{0, n}^{k_{0, n}}\, T_{n}^{v}\ \widehat{\simeq}\ \underset{i=0}{\overset{k}{\sum}}\, c_i\cdot t^i\, T^{v},\ {\rm where}\ k=\underset{i=0}{\overset{n}{\sum}}\, k_i,\ {\rm and}\ c_i\ {\rm coefficients\ for\ all}\ i.
\]
\end{lemma}

\begin{proof}
We prove Lemma~\ref{cform} by strong induction on the order of $w\, =\, t_{0, n}^{k_{0, n}}\, T_{n}^v$. The base of induction is $tt_1T_1^v$ which can be written in the closure as $T_1^vtt_1$. Applying the skein relations as shown in Figure~\ref{reducett} for the general case, results in a sum of the terms $T^v$ and $t^2T^v$ in $\mathcal{T}_3$, which, according to Definition~\ref{order}, are of lower order than $tt_1T_1^v$.

\smallbreak

Assume now that the result holds for all elements of less order than $w$. Then, for $w\, =\, t_{0, n}^{k_{0, n}}\, T_{n}^v$, apply first conjugation on the closure to bring the term $t_n^{k_n}\, T_n^v$ in the beginning of the monomial, i.e. $t_{0, n}^{k_{0, n}}\, T_{n}^v\ \widehat{=}\ \left( t^{k_n}\, T^v \right)\, t_{1, n}^{k_{0, n-1}}$. Rewrite the resulting monomial in the form $\left( t^{k_n}\, T^v \right)\, t_{1, n-2}^{k_{0, n-3}}\, \underline{t_{n-1}^{k_{n-2}}\, t_{n-1}^{k_{n-1}}}$ and apply the skein relations as shown in Figure~\ref{cflem}.
\begin{figure}[!ht]
\begin{center}
\includegraphics[width=5.8in]{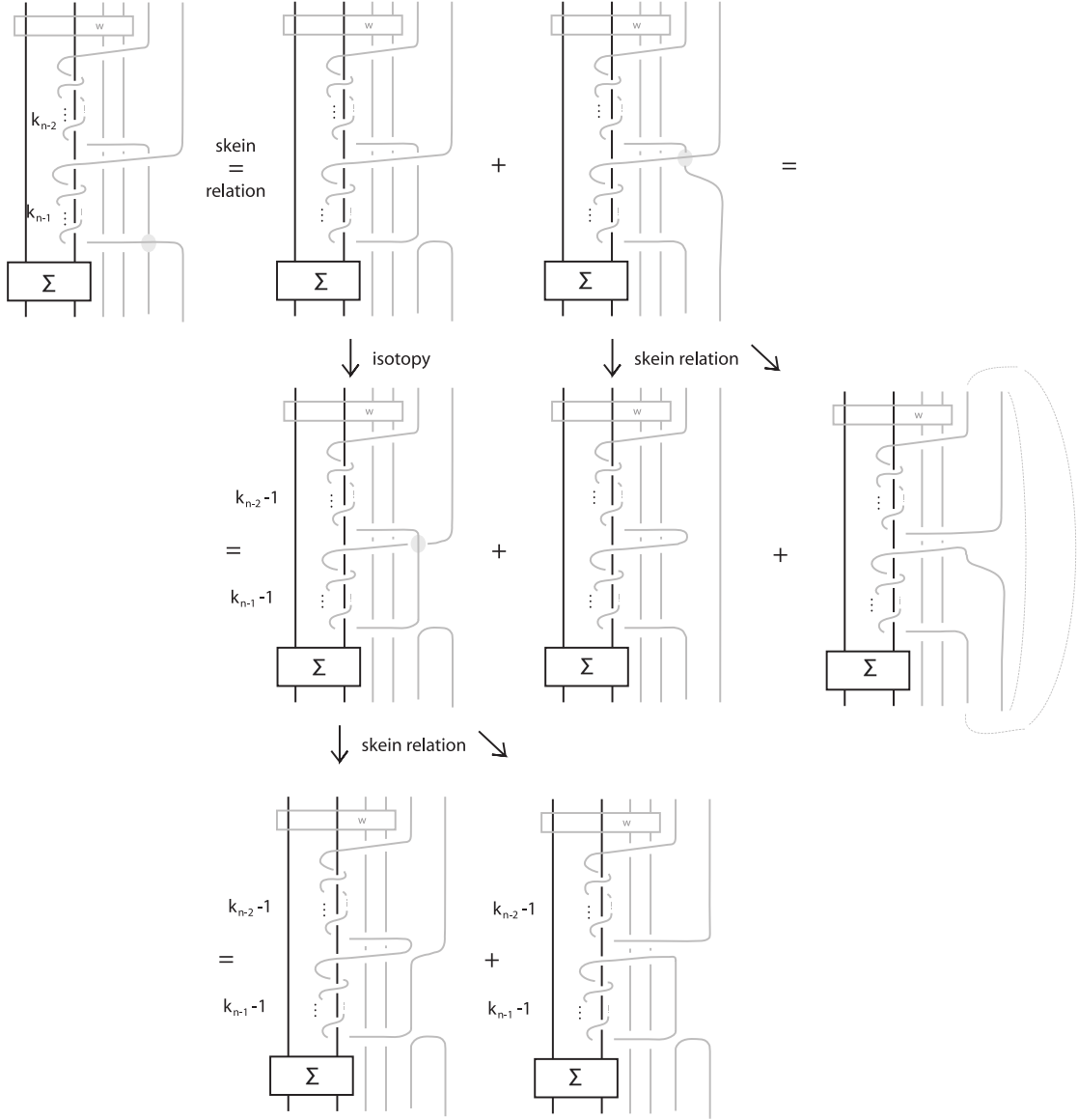}
\end{center}
\caption{The proof of Lemma~\ref{cform}.}
\label{cflem}
\end{figure}
We now apply conjugation on the closure of each resulting monomial in order to bring the term $t^{k_n}\, T^v$ in the end of each word. The result of this procedure is a sum of the following lower order than $w$ terms:
\[
\begin{array}{lcll}
u_1\, =\, t_{0, n-3}^{k_{0, n-3}}\, t_{n-2}^{k_{n-2}+k_{n-1}-2}\, t_{n-1}^{k_n}\, T_{n-1}^v & < & t_{0, n}^{k_{0, n}}\, T_{n}^v &,\ {\rm since}\ in(u_1)<in(w),\\
&&&\\
u_2\, =\, t_{0, n-3}^{k_{0, n-3}}\, t_{n-2}^{k_{n-2}-1}\, t_{n-1}^{k_{n-1}-1}\, t_n^{k_n}\, T_{n}^v & < & t_{0, n}^{k_{0, n}}\, T_{n}^v& ,\ {\rm since}\ k_{n-1}-1<k_{n-1},\\
&&&\\
u_3\, =\, t_{0, n-2}^{k_{0, n-3}}\, t_{n-2}^{k_{n-2}+k_{n-1}}\, t_{n-1}^{k_n}\, T_{n-1}^v & < & t_{0, n}^{k_{0, n}}\, T_{n}^v &,\ {\rm since}\ in(u_3)<in(w).
\end{array}
\]
The result follows from the induction hypothesis.
\end{proof}

An immediate consequence of Lemmas~\ref{aform}, \ref{bform} and \ref{cform}, is the following:

\begin{prop}
The set
\[
\mathcal{B}_{Tr^c}\ =\ \left\{ t^n\, T^v\ | \ n\in \mathbb{N}\cup\{0\},\ v\in \{0, \ldots, p\}\right\}
\]
\noindent spans the Kauffman bracket skein module of the complement of the $(2, 2p+1)$-torus knots.
\end{prop}

\subsection{Linear Independence of $\mathcal{B}_{T_{(2, 2p+1)}^c}$}\label{linind}

In this subsection we prove linear independence of $\mathcal{B}_{T_{(2, 2p+1)}^c}$, and thus, we show that the set $\mathcal{B}_{T_{(2, 2p+1)}^c}$ forms a basis for the Kauffman bracket skein module of the complement of the $(2, 2p+1)$-torus knots.

\smallbreak

Elements in $\mathcal{B}_{T_{(2, 2p+1)}^c}$ have no crossings in our setting and thus, conjugation, stabilization moves and skein relations cannot simplify the elements in $\mathcal{B}_{T_{(2, 2p+1)}^c}$ any further. Hence, it suffices to show that combed loop conjugations cannot isotope two different elements in $\mathcal{B}_{T_{(2, 2p+1)}^c}$. For that we introduce the {\it total winding} of elements in $\mathcal{B}_{T_{(2, 2p+1)}^c}$:

\begin{defn}\rm
Let $w=t^n\, T^v$, where $v\in \{0, \ldots, p\}$. We define the {\it total winding} of $w$, denoted by $win(w)$, to be the number $n+2\cdot v$.
\end{defn}

In other words, the total winding of an element in $\mathcal{B}_{T_{(2, 2p+1)}^c}$ counts the total number the moving strands twist around the two fixed strands. Obviously $win(t^n)=n$ and $win(T)=2$, since $T:=\tau\, t$.

\begin{prop}\label{twclc}
The total winding of an element $w$ in $\mathcal{B}_{T_{(2, 2p+1)}^c}$ is invariant under combed loop conjugations.
\end{prop}

\begin{proof}
Consider $w\in \mathcal{B}_{T_{(2, 2p+1)}^c}$. Then, we consider the following cases:
\[
\begin{array}{rclcl}
win\left(t^k \cdot w \cdot comb(t^{-k})\right) & = & win\left(t^k \cdot w \cdot (T^{-p-1}\, t\, T^{p})^k\right) & = &\\
&&&&\\
& = & k\ +\ win(w)\ +\ (-2pk-2k+k+2pk) & = & win(w),\\
&&&&\\
win\left(T^{\pm k} \cdot w \cdot comb(T^{\mp k})\right) & = & \pm 2k\ +\ win(w)\ \mp 2k & = & win(w),\\
&&&&\\
win\left(\tau^{k} \cdot w \cdot comb(\tau^{-k})\right) & = & win\left(\tau^k \cdot w \cdot (T^{-p}\, t^{-1}\, T^p)^k\right) & = & \\
&&&&\\
& = & k\ +\ win(w)\ +\ (-2pk-k+2pk) & = & win(w).
\end{array}
\]
\end{proof}

Thus, we have proved that for $n\neq m\in \mathbb{N}$, we have that $t^n\not\sim t^m$ and $t^n\, T^v\ \not\sim t^m\, T^v$. The only case that is not considered by the discussion above is the case $t^{n-2}T^{v+1}\not\sim t^n\, T^v  \not\sim t^{n+2}T^{v-1}$, since
\[
win(t^{n-2}T^{v+1})\ =\ win(t^n\, T^v)=n+2v=win(t^{n+2}T^{v-1}).
\]
\noindent This follows from the fact that combed loop conjugations cannot alter the exponent of a $T$-looping generator whose exponent is less than or equal to $p$ (recall Eq.~\ref{clceq}).

\bigbreak

We have proved the main result of this paper:

\begin{thm}
The set
\[
\mathcal{B}_{Tr^c}\ =\ \left\{ t^n\, T^v\ | \ n\in \mathbb{N}\cup\{0\},\ v\in \{0, \ldots, p\}\right\}
\]
\noindent is a basis for $KBSM(T_{(2, 2p+1)}^c)$.
\end{thm}

\section{On the Kauffman bracket skein module of 3-manifolds obtained by surgery along the trefoil knot}\label{trefc}

In this section we first demonstrate the braid method for computing the Kauffman bracket skein module of $(2, 2p+1)$-torus knots for the case of the complement of the trefoil knot, $Tr^c$ and we also discuss further steps needed for computing Kauffman bracket skein modules of c.c.o. 3-manifolds that are obtained from $S^3$ by surgery along the trefoil knot.

\subsection{The Kauffman bracket skein module of the complement of the trefoil knot}

We now demonstrate the method described in this paper for the computation of the Kauffman bracket skein module of the complement of the trefoil knot, $KBSM(Tr^c)$, via braids. Note that the fixed part of a mixed braid that represents $Tr^c$ is $\Sigma_1^{3}$.

\smallbreak

We first give algebraic expressions for combed loop conjugations (recall Theorem~\ref{algcco}). 

\begin{equation}\label{clctkc1}
\begin{array}{lclcl}
comb(t)\ =\ T^{-1}\, t^{-1}\, T^{2} & , & comb(\tau)\ =\ T^{-1}\, t\, T & , & comb(T^{\pm 1})\ =\ T^{\pm 1}.
\end{array}
\end{equation}

An immediate result of Eq.(\ref{clctkc1}) is that the set $\mathcal{T}_1$ spans $KBSM(Tr^c)$, since $\tau\, \sim\, t$. We now apply Lemma~\ref{lsq} in order to deal with the $T$-generators in the case of $Tr^c$. We do that by ``grouping'' the $T$-generators into pairs as follows:

\begin{equation}\label{Tsq1}
t_{0, n}^{\prime}\, T_{n+1, n+m}^{\prime}\ =\ \begin{cases} \underset{\phi=0}{\overset{\frac{m}{2}}{\sum}}\, \left[t_{0, n}^{\prime}\cdot \left({T_{n+1, n+1+ \phi}^{\prime}}\right)^2\right] &,\ {\rm for}\ m\ {\rm even}\\
&\\
\underset{\phi=0}{\overset{\frac{m-1}{2}}{\sum}}\, \left[t_{0, n}^{\prime}\, T_{n+1}^{\prime}\cdot \left({T_{n+2, n+1+ \phi}^{\prime}}\right)^2\right] &,\ {\rm for}\ m\ {\rm odd}\\
\end{cases}
\end{equation}

Equivalently, we have that the set $\mathcal{T}r_1$ that consists of elements in the form of Equation~(\ref{Tsq1}) is a spanning set of $KBSM(Tr^c)$. We now proceed by reducing the $T_i^{\prime}$'s in elements in the set $\mathcal{T}r_1$. This is done in two steps. Let $w$ be a monomial in the $t_i^{\prime}$'s and $T_i^{\prime}$'s such that $in(w)=n$. Then, the following holds in $KBSM(Tr^c)$:
\begin{equation}\label{trl1}
w\, {T_{n+1}^{\prime}}^2\ \widehat{\simeq}\ w\, {t_{n+1}^{\prime}}^2\, T_{n+1}^{\prime}
\end{equation}

\noindent The relation follows as a result of combed loop conjugation. Indeed we have the following:
\[
w\, {T_{n+1}^{\prime}}^2\ \sim\ w\, t_{n+1}^{\prime}\, T_{n+1}^{\prime}\, \left({T_{n+1}^{\prime}}^{-1}\, {t_{n+1}^{\prime}}^{-1}\, {T_{n+1}^{\prime}}^2 \right)\ =\ w\, t_{n+1}^{\prime}\, T_{n+1}^{\prime}\, comb(t_{n+1}^{\prime})\ \widehat{\simeq}\ w\, {t_{n+1}^{\prime}}^2\, T_{n+1}^{\prime}.
\]

\noindent Moreover, we have that:

\begin{equation}\label{trl2}
w\, {t_{n+1}^{\prime}}^p\, {T_{n+1}^{\prime}}^2 \ \widehat{\simeq}\ w\, {t_{n+1}^{\prime}}^{p+2}\, {T_{n+1}^{\prime}},\ \ {\rm since}
\end{equation}

\[
\begin{array}{lcl}
w\, {t_{n+1}^{\prime}}^p\, {T_{n+1}^{\prime}}^2 & \sim & w\, {t_{n+1}^{\prime}}^p\, t_{n+1}^{\prime}\, T_{n+1}^{\prime}\, \left({T_{n+1}^{\prime}}^{-1}\, {t_{n+1}^{\prime}}^{-1}\, {T_{n+1}^{\prime}}^2 \right)\ :=\\
&&\\
& := & w\, {t_{n+1}^{\prime}}^p\, t_{n+1}^{\prime}\, T_{n+1}^{\prime}\, comb(t_{n+1}^{\prime})\ \widehat{\simeq}\ w\, {t_{n+1}^{\prime}}^{p+2}\, {T_{n+1}^{\prime}}.
\end{array}
\]

\bigbreak

We now demonstrate how Relations~(\ref{trl1}) and (\ref{trl2}) affect elements in $\mathcal{T}r_1$. In the generic example that follows, we underline expressions which are crucial for the next step.

\begin{ex}
\[
\begin{array}{lcl}
t_{0, n}^{\prime}\, T_{n+1, n+m}^{\prime} & = & t_{0, n}^{\prime}\, T_{n+1, n+m-2}^{\prime}\, \underline{T_{n+m-1}^{\prime}\, T_{n+m}^{\prime}}\ \overset{Lemma~\ref{grT}}{\widehat{\simeq}}\\
&&\\
& \widehat{\simeq} & -A^2\, t_{0, n}^{\prime}\, T_{n+1, n+m-2}^{\prime}\, \underline{{T_{n+m-1}^{\prime}}^2}\, +\, A\, (-A^2-A^{-2})\, t_{0, n}^{\prime}\, T_{n+1, n+m-2}^{\prime}\ \overset{Lemma~\ref{l1}}{\widehat{\simeq}}\\
&&\\
& \widehat{\simeq} & -A^2\, t_{0, n}^{\prime}\, T_{n+1, n+m-3}^{\prime}\, \underline{T_{n+m-2}^{\prime}\, {t_{n+m-1}^{\prime}}^2\, {T_{n+m-1}^{\prime}}}\ + \\
&&\\
& + & A\, (-A^2-A^{-2})\, t_{0, n}^{\prime}\, T_{n+1, n+m-2}^{\prime}\ \overset{conj.}{\widehat{\simeq}}\\
&&\\
& \widehat{\simeq} & -A^2\, t_{0, n}^{\prime}\, T_{n+1, n+m-3}^{\prime}\, {t_{n+m-2}^{\prime}}^2\, \underline{T_{n+m-2}^{\prime}\,  {T_{n+m-1}^{\prime}}}\ +\\
&&\\
& + & A\, (-A^2-A^{-2})\, t_{0, n}^{\prime}\, T_{n+1, n+m-2}^{\prime}\ \overset{Lemma~\ref{grT}}{\widehat{\simeq}}\\
&&\\
& \widehat{\simeq} & A^4\, t_{0, n}^{\prime}\, T_{n+1, n+m-3}^{\prime}\, {t_{n+m-2}^{\prime}}^2\, {T_{n+m-2}^{\prime}}^2\\
&&\\
& - & -\ A^3\, (-A^2-A^{-2})\, t_{0, n}^{\prime}\, T_{n+1, n+m-3}^{\prime} \, {t_{n+m-2}^{\prime}}^2\ +\\
&&\\
& + & A\, (-A^2-A^{-2})\, t_{0, n}^{\prime}\, T_{n+1, n+m-2}^{\prime}\ \sim \ \ldots\\
\end{array}
\]
\end{ex}

Continuing that way, we will eventually obtain a sum of elements of the following forms:

\begin{equation}\label{bel1}
\begin{array}{cll}
(a)-form: & t_{0, n}^{\prime}\, {t_{n+1, n+r_1}^{\prime}}^{\lambda_{n+1, n+r_1}}\\
&&\\
(b)-form: & t_{0, n}^{\prime}\, {t_{n+1, n+r_2}^{\prime}}^{\mu_{n+1, n+r_2}}\, T_{n+r_2}^{\prime}\\
&&\\
(c)-form: & t_{0, n}^{\prime}\, {t_{n+1, n+r_3}^{\prime}}^{\phi_{n+1, n+r_3}}\, T_{n+r_3+1}^{\prime},
\end{array}
\end{equation}

\noindent where $n, r_1, r_2, r_3, \lambda_i, \mu_i, \phi_i\in \mathbb{N}$ for all $i$. Thus, we have reduced the number of the $T$-loop generators needed in order to generate $KBSM(Tr^c)$ and we have showed that the set $\mathcal{T}r_2$ that consists of elements in the (a), (b) and (c)-forms of Eq.~(\ref{bel1}), spans the Kauffman bracket skein module of the complement of the trefoil knot, $KBSM(Tr^c)$.

\smallbreak

Finally, using the ordering relation of Definition~\ref{order}, we express elements in $\mathcal{T}r_2$ to sums of elements in 

\[
\mathcal{B}_{Tr^c}\ =\ \left\{t^n\, T^j\ |\ n\in \mathbb{N}\cup \{0\},\ j\in \{0, 1\} \right\}.
\]

\noindent Linear independence of $\mathcal{B}_{Tr^c}$ follows as in \S~\ref{linind}.

\subsection{On the Kauffman bracket skein module of surgery along the trefoil knot}\label{trefcs}

In this subsection we deal with the Kauffman bracket skein module of a c.c.o. 3-manifold, $M$, obtained from $S^3$ by surgery along the trefoil knot. Our technique can be applied for the cases of integral and rational surgery and for 3-manifolds obtained by surgery along any $(m, r)$-torus knot also. It is worth mentioning now that a manifold obtained by rational surgery from $S^3$ along an $(2,2p+1)$-torus knot with rational coefficient $r/q$ is either the lens space $L(|q|,r(2p+1)^2)$, or the connected sum of two lens spaces $L(2,2p+1)\sharp L(2p+1,2)$, or a Seifert manifold (for more details see \cite{Mo}). Our intention is to describe the main ideas of the ``braid approach'' toward the computation of the Kauffman bracket skein module of $M$, where $M$ is obtained from $S^3$ by integral surgery, $r$, along the trefoil knot. In a sequel paper we shall compute $KBSM(M)$ in the general case of surgery along $(2,2p+1)$-torus knots. We now recall some results from \cite{DL1}.

\bigbreak

Let $K$ be an oriented link in $M$, a 3-manifold $M$ obtained from $S^3$ by surgery along the trefoil knot. We fix $\widehat{B}$ pointwise on its projection plane and we represent $K$ by a mixed link in $S^3$ as before. The difference from $Tr^c$ lies in isotopy. More precisely, surgery along $\widehat{B}$ is realized in two steps: the first step is to consider the complement $S^3\backslash \widehat{B}$ of the trefoil knot, and then attach to it a solid torus according to its surgery description. Thus, we view isotopy in $M$ as isotopy in $S^3\backslash \widehat{B}$ together with the {\it band moves} in $S^3$, which are similar to the second Kirby move. Note also that there are two types of band moves, but as shown in \cite{DL1}, only one of the two types is needed in order to formulate the analogue of the Reidemeister theorem of knots and links in $M$. For an illustration for the case of the 3-manifold obtained from $S^3$ by rational surgery $2/3$ along the trefoil knot see Figure~\ref{bmfig}, where the performance of the 2/3 band move is denoted by $(2, 3)$-b.m..

\begin{figure}[!ht]
\begin{center}
\includegraphics[width=5.3in]{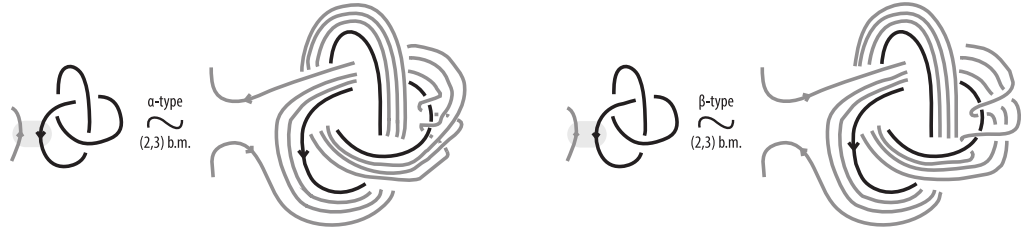}
\end{center}
\caption{The two types of band moves in the case of $2/3$-surgery.}
\label{bmfig}
\end{figure}

Thus, in order to compute $KBSM(M)$, it suffices to consider elements in a basis of $Tr^c$ and study the effect of the band moves on them. With a little abuse of notation we write: 
\[
KBSM(M)\ =\ \frac{KBSM(Tr^c)}{<band\ moves>}.
\]

Our starting point is the basis $\mathcal{B}_{Tr^c}$ that consists of elements of the form $\widehat{t^n\, T^j}$, where $n\in \mathbb{N}\cup \{0\}$ and $j\in \{0, 1\}$, and apply band moves in order to deduce them into elements in a spanning set of $KBSM(M)$. For convenience, we will be using both types of band moves to translate isotopy of knots and links in $M$ to isotopy of mixed links in $S^3$. For the type $\beta$ band moves for the case of integral surgery along the trefoil knot in open braid form see Figure~\ref{intsurgtref}, where after the performance of the band move, we also perform the technique of parting on the resulting mixed link.

\begin{figure}[!ht]
\begin{center}
\includegraphics[width=5in]{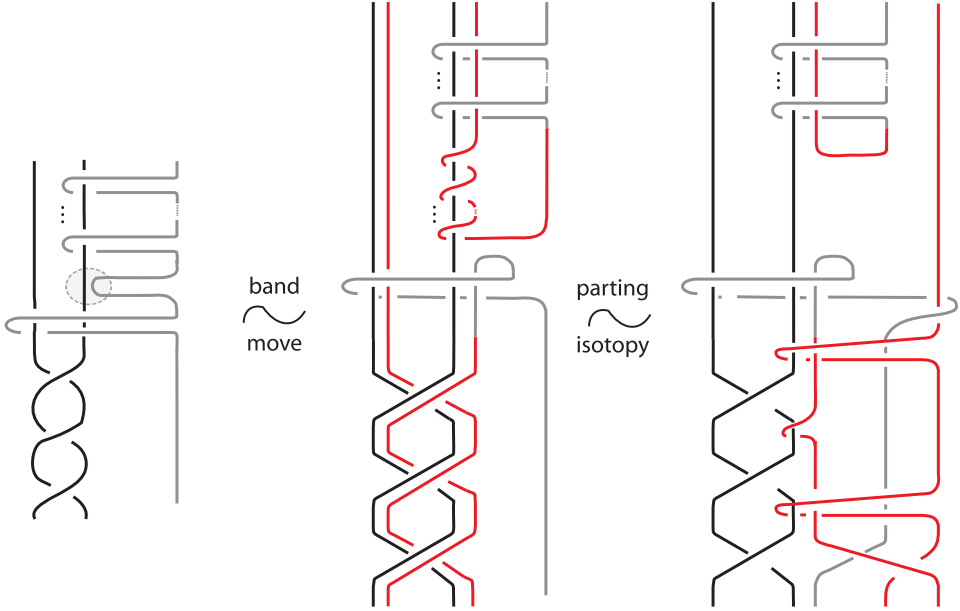}
\end{center}
\caption{Band moves and parting.}
\label{intsurgtref}
\end{figure}

The type $\alpha$ band move is translated on the level of mixed braids by the {\it geometric braid band moves}. A {\it geometric braid band move} is a move between geometric mixed braids which is a band move between their closures. It starts with a little band oriented downward, which, before sliding along a surgery strand, gets one twist {\it positive\/} or {\it negative\/}.

\begin{figure}[!ht]
\begin{center}
\includegraphics[width=3in]{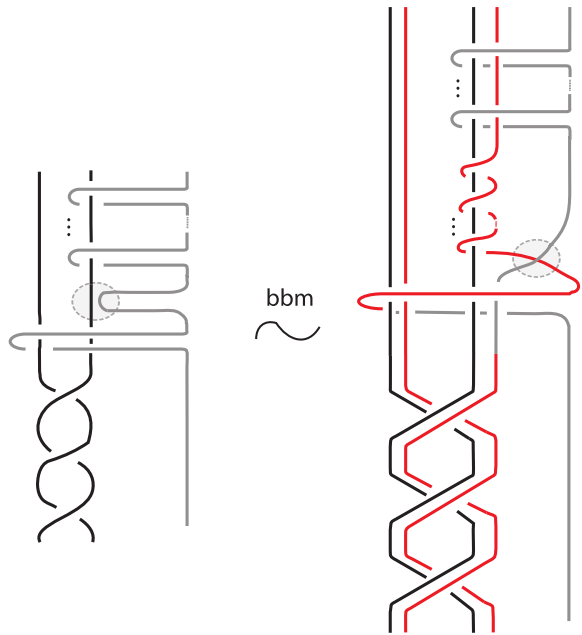}
\end{center}
\caption{Braid band moves.}
\label{intsurgtref1}
\end{figure}

After the performance of a bbm on a mixed braid we perform the technique of parting (see Figure~\ref{intsurgtref} for the case of integral surgery along the trefoil knot) and the technique of combing on the resulting mixed braid, in order to obtain an algebraic mixed braid followed by its ``coset'' part which corresponds to the description of the manifold $M$. 

\bigbreak

\begin{remark}\rm
Braid band moves (or bbm's) on mixed braids reflect braid equivalence on $M$ and as shown in \cite{LR1}, they may be always assumed to take place at the top part of a mixed braid and on the right of the specific surgery strand (\cite{LR2} Lemma~5) up to braid equivalence. In a similar way it follows that the place of the fixed part of a mixed link where the band move is performed is irrelevant (up to isotopy) and we will be considering band moves occurring at the top part of the fixed part of a mixed link in the case of the $\beta$-type bbm's, and after the $t$-generator in the case of the $\alpha$-type bbm's. In that way we obtain homogeneous algebraic expressions for bbm's.
\end{remark}

The algebraic expressions of the braid band moves (recall Figure~\ref{intsurgtref1}) are:

\[
t^n\, T^j\ \overset{r-bbm}{\underset{r\in \mathbb{Z}}{\sim}}\ t^r\,  \sigma_1^{\pm 1}\, t_1^{n}\, \sigma_2^2\, T_1^j\,\, comb,
\]

\noindent where $t_1\, =\, \sigma_1\, t\, \sigma_1$, $T_1\, =\, \sigma_1\, T\, \sigma_1$ and $comb$ corresponds to the combing of the new moving strands appearing from the performance of the bbm and that has the following algebraic expression:

\[
comb\ =\ \left(T_1\, \sigma_2^2\right)\, \left(\sigma_2\, t_1\, \sigma_2^{-1} \right)\, \left(t^{-1}\, T \right)\, \left(\sigma_2\, \sigma_1\, T^{-1}\, t\, \sigma_1\, \sigma_2^{-1} \right)\, \sigma_1^{-1}\, \sigma_2^{-1}.
\]

Applying the Kauffman skein relations, we cancel the braiding generators and obtain elements in $B_{Tr^c}$, where the total winding number increases by $r+3$.

\smallbreak

Consider now the performance of a $\beta$-type band move (recall Figure~\ref{intsurgtref}). After the performance of this type of band move and using isotopy in $Tr^c$, we may cancel the looping generator $t^n$ of the initial element and obtain $t^{n-r}$. Since the $comb$ part remains the same as before, we conclude that the performance of the band move will result in decreasing the total winding number by $3-r$.

\bigbreak

Thus, we may always reduce the elements in $B_{Tr^c}$ to elements in a finite subset of $B_{Tr^c}$, keeping control of the exponent of the $t$-generator as follows: if the starting element in $B_{Tr^c}$ is $t^n\, T^j$, where $n>r$, then perform a $\beta$-type band move to reduce the exponent of $t$ to $n-r+3$ and we consider the image of the result of the performance of the band move in $KBSM(Tr^c)$. We continue in that way if needed in order to reduce the exponent $n$ to an exponent in a finite subset of $\mathbb{N}$. The case where $r=0$ is of special interest, since the band moves will affect elements $w$ in $B_{Tr^c}$ in the same way, i.e., they only increase the total winding of $w$ by 3 and further research is required to understand how this affects monomials in $B_{Tr^c}$.

\end{document}